\documentclass[11pt]{amsart}

\usepackage{amssymb,amsfonts}
\usepackage{graphicx}
\usepackage[all,arc]{xy}
\usepackage{enumitem}
\usepackage{hyperref}
\usepackage{mathrsfs}
\usepackage{tikz-cd}
\usepackage{todonotes}
\usetikzlibrary{shapes.geometric}
\usepackage{mathtools}
\usepackage[left=1in,top=1in,right=1in,bottom=1in]{geometry}

\usepackage[backend=bibtex,style=alphabetic,doi=false,isbn=false,url=false,giveninits=true,maxbibnames=50]{biblatex}
\usepackage{comment}
\DeclareMathOperator{\diag}{diag}
\DeclareMathOperator{\gr}{gr}

\newtheorem{thm}{Theorem}[section]
\newtheorem{mainresult}{Theorem}
 
\newtheorem*{thm*}{Theorem}

\newtheorem{cor}[thm]{Corollary}
\newtheorem{prop}[thm]{Proposition}
\newtheorem{lem}[thm]{Lemma}

\theoremstyle{definition}
\newtheorem{defn}[thm]{Definition}

\newtheorem{rem}[thm]{Remark}
\newtheorem{claim}[thm]{Claim}

\theoremstyle{remark}

\newcommand{\N}{\mathbb{N}}
\newcommand{\Z}{\mathbb{Z}}

\newcommand{\e}{\varepsilon} 
\renewcommand{\d}{\delta}
\newcommand{\gl}{\mathfrak{gl}}
\newcommand{\g}{\mathfrak{g}}
\renewcommand{\P}{\mathbb{P}}

\newcommand{\C}{\mathbb{C}}

\newcommand{\A}{\mathbb{A}}

\renewcommand{\emptyset}{\varnothing}

\makeatletter
\let\c@equation\c@thm
\makeatother
\numberwithin{equation}{section}

\newcommand{\Sym}{\mathrm{Sym}}

\title{FFLV}
\title{FFLV bases for covariant representations of $\mathfrak{gl}(m|n)$}

\begin{document}
\author[Ahmad]{Ibrahim Ahmad}
\address[Ahmad]{Chair of Algebra and Representation Theory, RWTH Aachen University, Pontdriesch 10-16, 52062 Aachen, Germany}
\email{ahmad@art.rwth-aachen.de}
\author[Enugandla]{Pranav Enugandla}
\address[Enugandla]{University of California, Berkeley, CA, United States of America}
\email{shreepranav\_varma@berkeley.edu}
\maketitle
\begin{abstract}
    
    We study the PBW filtration on covariant representations for the Lie superalgebra $\mathfrak{gl}(m|n)$. We prove for all covariant weights of the form $(\lambda|\mu,0^{n-1})$, that there exists a lattice polytope such that the lattice points of this polytope parametrize a basis of the corresponding associated graded space. As a consequence, we obtain degenerations of partial flag supervarieties for the supergroup $GL(m|n)$ into toric supervarieties.

\end{abstract}
\section{Introduction}

Let $\g=\mathfrak{n}^-\oplus \mathfrak{h}\oplus\mathfrak{n}^+$ be a finite-dimensional simple Lie algebra with a fixed triangular decomposition. 
By the PBW-theorem, the associated graded space of the universal enveloping algebra of $\mathfrak{n}^-$, which we denote by $\mathcal{U}(\mathfrak{n}^-)$, is the symmetric algebra $\Sym(\mathfrak{n}^-)$. This filtration induces an abelianzed action of $\Sym(\mathfrak{n}^-)$ on the associated graded module $\gr V(\lambda)=:V(\lambda)^a$. 
In \cite{FFLV_Type_A} and \cite{FFLV_Type_C}, Feigin, Fourier and Littelmann studied the PBW filtration of the finite-dimensional irreducible $\g$-modules $V(\lambda)$ in types $\mathsf{A}_n$ and $\mathsf{C}_n$, and constructed a monomial basis of $V(\lambda)^a$ parametrized by the lattice points of a normal polytope, which is known as the \textit{Feigin-Fourier-Littelmann-Vinberg polytope}. Similar results on polytopes parameterizing monomial bases of the associated graded modules are known in type $\mathsf{G}_2$ \cite{G_2-Result} and $\mathsf{B}_n$ \cite{Kus-Type-B,Makhlin-Type-B}.
The FFLV polytopes are fascinating in a variety of aspects.

From a combinatorial standpoint, these polytopes are very closely linked to the Gelfand-Tsetlin polytopes \cite{GT}: Ardila, Bliem and Salazar proved in \cite{ABS} that both polytopes can be realized as marked versions of Stanley’s order and chain polytopes \cite{Stanley} and provide piecewise-affine bijection between them.

Geometrically, Feigin introduced the PBW-degnerate flag variety in \cite{Feigin-GaM} by considering the closure of the highest weight orbit of the abelianized Lie group $\overline{N^{-,a}\cdot [v_\lambda]}\subseteq\P V(\lambda)^a$ and showed that there is a flat degeneration of the flag variety into the PBW-degenerate flag variety. 
The normality of the FFLV polytopes allowed Feigin, Fourier and Littelmann to conclude in \cite{FFL-Favourable} that the representations $V(\lambda)$ are \textit{favourable}. 
This condition provides a way to degenerate the embedded partial flag variety $G/P\hookrightarrow\P V(\lambda)$ to the PBW-degenerate flag variety and also degenerate the PBW-degenerate flag variety to the toric variety given by the FFLV polytope. 
Later on, Fang, Fourier and Littelmann gave in \cite{Birational_Seq} a more uniform construction of toric degenerations of flag varieties via \textit{birational sequences}.

Previously, Fourier and Kus constructed in \cite{Fourier-Kus} lattice points of a polytope yielding monomial bases for \textit{typical} representations of various  Lie superalgebras, including $\gl(m|n)$. In this article, we provide lattice polytopes parametrising monomial bases compatible with the PBW-filtration for certain \textit{covariant} $\gl(m|n)$-representations.
First introduced by Berele-Regev in \cite{Berele-Regev-Covariant} and Sergeev in \cite{Sergeev-Covariant}, covariant representations form the irreducible components of the semisimple representation $(\C^{m|n})^{\otimes k}$ for $k\geq 0$. 
Their highest weights are parametrized by pairs of partitions $(\lambda|\mu)=(\lambda_1,\dots,\lambda_m|\mu_1,\dots,\mu_n)$ such that $\lambda_m\geq\left|\left\{ i\mid \mu_i\neq0\right\}\right|$. 
This can be translated combinatorially as saying that the composition $(\lambda,\mu’)$, where $\mu’$ is the transpose of $\mu$, is a partition whose Young diagram sits inside the $(m,n)$-hook, i.e., does not contain the $(n+1)^\text{th}$ box in row $m+1$. 
For covariant representations, there are already Gelfand-Tsetlin bases available constructed in \cite{Molev} and \cite{Stoilova-Jeugt}, but an FFLV counterpart, even in spirit, remained missing.

We show that for all diagrams $(\lambda,\mu')$ in the $(m,n)$-hook with $\mu'=(1^k)$ for some $k\geq 0$, we can construct an FFLV-type basis. 
In particular, we get the first class of atypical representations having such a basis.
Moreover, in the typical case, our objects agree with the polytopes and bases from \cite[Theorem 1.1]{Fourier-Kus}. Our first set of results is the following; we refer to section \ref{sec-hook} for the precise definition.
\begin{mainresult}
    Let $\lambda$ be a partition of length $l(\lambda)\leq m$ with $\lambda_m\neq0$ and $\mu\geq0$. Then there exists a lattice polytope $P(\lambda|\mu,0^{n-1})\subseteq\mathbb{R}^{d_{1}}\times[0,1]^{d_{2}}$ with the following properties:
    \begin{enumerate}
        \item{
            The lattice points $S(\lambda|\mu,0^{n-1})$ parametrise a monomial basis of $\gr V(\lambda|\mu,0^{n-1})$.
        }
        \item{
            We have the following Minkowski decomposition
            \begin{equation*}
                S(\lambda|\mu,0^{n-1})=(S(\lambda-1^m|0^n)+S(1^m|\mu,0^{n-1}))\cap\{0,1\}^{d_2}
            \end{equation*}
        }
    \end{enumerate}
\end{mainresult}

If we restrict ourselves to the diagrams living inside the horizontal strip of the $(m,n)$-hook, i.e. the diagrams with at most $m$ rows, then we get the following further results. For precise definitions, see sections \ref{sec-defn-polytope}, \ref{sec-indep-strip} and \ref{sec-degen}.
\begin{mainresult}
    Let $\lambda,\mu$ be partitions of lengths $l(\lambda),l(\mu)\leq m$. Then
    \begin{enumerate}
        \item{
            We have the Minkowski decomposition
            \begin{equation*}
                S(\lambda|0^n)+S(\mu|0^n)=S(\lambda+\mu|0^n)\cap\{0,1\}^{d_{2}}
            \end{equation*}
        }
        \item{
            The module $V(\lambda|0^{n})$ is favourable with respect to the monomial order in Definition \ref{def-monomial-order}.
        }
        \item{
            Let $G/P\hookrightarrow G\cdot[v_{(\lambda|0^n)}]\subseteq\mathbb{P}(V(\lambda|0^n))$ be the embedded partial flag supervariety as the highest weight orbit for the supergroup $G=GL(m|n)$. There exists a flat proper morphism of superschemes $\kappa:\mathrm{Proj}\mathcal{R}\rightarrow\mathbb{A}^{1|0}$ such that the fibers are
            \begin{equation*}
                \kappa^{-1}(a)\cong\begin{cases}
                    \mathrm{Proj}(\C[G/P]_{(\lambda|0^n)}) &a\neq0\\
                    \mathrm{Proj}(\C[x^\mathbf{s}u\,|\,\mathbf{s}\in S(\lambda|0^n)]) &a=0,
                \end{cases}
            \end{equation*}
            where $\C[G/P]_{(\lambda|0^n)}$ is the homogeneous coordinate ring of the partial flag supervariety.
            Moreover, the affine spectrum of the algebra of the special fibre $\mathrm{Spec}\,\C[x^\mathbf{s}u\,|\,\mathbf{s}\in S(\lambda|0^n)]$ is a toric supervariety as defined by Jankowski \cite{Toric-Super-New,Toric-Super-One-Odd}.
        }
    \end{enumerate}
\end{mainresult}

This paper is organized as follows: In section \ref{sec-prelim}, we set the notation and introduce preliminaries. In section \ref{sec-defn-polytope}, we provide the combinatorial model for our polytope for weights of the form $(\lambda|0^n)$ and state its defining inequalities. In section \ref{sec-spanning-strip}, we show show that our set of monomials is a spanning set for $\gr V(\lambda|0^n)$. In section \ref{sec-indep-column-strip} we show that they are linearly independent for exterior powers $\bigwedge^k\C^{m|n}$ and in section \ref{sec-indep-strip} we show the Minkowski decomposition and conclude the linear independence for general $\gr V(\lambda|0^n)$. In section \ref{sec-hook} we state our polytope for the weights of the form $(\lambda|\mu,0^{n-1})$ and prove our results. In section \ref{sec-degen} we show that we get the degeneration of the partial flag supervariety.

\noindent
\textbf{Acknowledgement:}
Both authors would like to thank their advisors, Ghislain Fourier and Vera Serganova, as well as Cameron Chang, Evgeny Feigin and Xin Fang for helpful discussions. Further, they want to thank Igor Makhlin for pointing out an error in a previous version of this article. IA would like to thank the German-American Fulbright Commission for awarding a Fulbright scholarship for a research visit to UC Berkeley during which this article was written. Its contents are solely the responsibility of the authors and do not necessarily represent the official views of the Fulbright Program, the Government of the United States, or the German-American Fulbright Commission.
\section{Preliminaries}\label{sec-prelim}
We begin by recalling the notions on $\gl(m|n)$, which can be found in \cite{Fioresi,Cheng-Wang,Serg_Rep}. Many standard facts in \cite{Cheng-Wang} originally stem from Kac's seminal Advances paper \cite{Kac_Lie} and his succeeding article \cite{Kac_Rep}.\smallskip

The Lie superalgebra $\gl(m|n)$ is given as a set by
\begin{equation*}
    \gl(m|n)=\left\{\begin{pmatrix}
        A&B\\
        C&D
    \end{pmatrix}\ \middle|\  A\in\C^{m\times m},D\in\C^{n\times n}, B\in\C^{m\times n},C\in\C^{n\times m}\right\}.
\end{equation*}

The even part, which we denote by $\gl(m|n)_{\overline{0}}$, consists of all block diagonal matrices, and the odd part, which we denote by $\gl(m|n)_{\overline{1}}$, consists of all block antidiagonal matrices.

The bracket structure is given by
\begin{equation*}
    \left[X,Y\right]=X\cdot Y-(-1)^{|X|\cdot|Y|}Y\cdot X,
\end{equation*}
for $X\in\gl(m|n)_{\overline{i}}$, $Y\in\gl(m|n)_{\overline{j}}$ both homogeneous and $|X|=\overline{i}$ and $|Y|=\overline{j}$, respectively.\smallskip

On $\gl(m|n)$, we fix the \textit{Cartan subalgebra} $\mathfrak{h}\subseteq\gl(m|n)$ to be the set of diagonal matrices.

Furthermore, we fix our Borel subalgebra $\mathfrak{b}\supseteq\mathfrak{h}$ to be the \textit{distinguished Borel subalgebra} consisting of all upper triangular matrices in $\gl(m|n)$.

Hence, we get the triangular decomposition
\begin{equation*}
    \gl(m|n)=\mathfrak{n}^-\oplus\mathfrak{h}\oplus\mathfrak{n}^+=\mathfrak{n}^-\oplus\mathfrak{b},
\end{equation*}
where $\mathfrak{n}^-$ (resp. $\mathfrak{n}^+$), refers to the set of strictly lower (resp. upper) triangular matrices of $\gl(m|n)$.
\begin{rem}
    Note that our choice of Borel is a deliberate one as, in contrast to the classical case, the Borel subalgebras of $
    \gl(m|n)$ are not conjugate to each other.
\end{rem}
\paragraph{\textit{Root system}}
The \textit{roots} of $\gl(m|n)$ are those functionals $\alpha\in\mathfrak{h}^*\setminus\{0\}$ such that the \textit{root space}
\begin{equation*}
    \mathfrak{gl}(m|n)_\alpha:=\left\{x\in\gl(m|n)\ \middle|\ \left[h,x\right]=\alpha(h)x,\,\forall h\in\mathfrak{h}\right\}
\end{equation*}
is non-zero.

For $h=\diag(a_1,\dots,a_m,b_1,\dots,b_n)\in\mathfrak{h}$, we define the functionals
\begin{equation*}
    \e_i(h)=a_i,\hspace{1em}\delta_j(h)=b_j,\hspace{1em}\text{ for }1\leq i\leq m,\,1\leq j\leq n
\end{equation*}

The set of roots of $\gl(m|n)$, denoted $\Phi=\Phi_{\overline{0}}\cup \Phi_{\overline{1}}$, consists of the following set of functionals
\begin{align*}
    \Phi_{\overline{0}}&=\left\{\e_i-\e_j,\,\d_r-\d_s\middle|1\leq i\neq j\leq m,\,1\leq r\neq s\leq n\right\}\\
    \Phi_{\overline{1}}&=\left\{\pm(\e_i-\d_j)\middle|1\leq i\leq m,\,1\leq j\leq n\right\}.
\end{align*}
We call the set of roots in $\Phi_{\overline{0}}$ (resp. $\Phi_{\overline{1}}$) \textit{even} (resp. \textit{odd}), as their root spaces have the respective parity.

More explicitly, we have 
\begin{align*}
    \gl(m|n)_{\e_i-\e_j}=\langle E_{i,j}\rangle_\C,\hspace{1em}&\gl(m|n)_{\d_r-\d_s}=\langle E_{m+r,m+s}\rangle_\C\\
    \gl(m|n)_{\e_i-\d_j}=\langle E_{i,m+j}\rangle_\C,\hspace{1em}&\gl(m|n)_{\d_j-\e_i}=\langle E_{m+j,i}\rangle_\C
\end{align*}

With respect to the distinguished Borel subalgebra $\mathfrak{b}$, we have the following set of \textit{positive roots} $\Phi^+=\Phi^+_{\overline{0}}\cup\Phi^+_{\overline{1}}$, where
\begin{align*}
    \Phi^+_{\overline{0}}&=\left\{\e_i-\e_j,\,\d_r-\d_s\,\middle|\,1\leq i< j\leq m,\,1\leq r< s\leq n\right\}\\
    \Phi^+_{\overline{1}}&=\left\{\e_i-\d_j\,\middle|\,1\leq i\leq m,\,1\leq j\leq n\right\}.
\end{align*}

Having now fixed the positive roots, we then define the \textit{negative root vector} (resp. \textit{positive root vector}), $f_\alpha\in\mathfrak{n}^-$ (resp. $e_\alpha\in\mathfrak{n}^+$), for a given positive root $\alpha\in\Phi^+$ to be the elementary matrix spanning the root space $\gl(m|n)_{-\alpha}$ (resp. $\gl(m|n)_{\alpha}$).

\begin{rem}\label{rem-barred-indices-and-defn-set-I}
    To better distinguish between even and odd indices, we introduce the set
    \begin{equation*}
        K:=[m]\cup\overline{[n]},
    \end{equation*}
    where $[m]=\{1,\dots,m\}$ and $\overline{[n]}=\left\{\overline{1},\dots,\overline{n}\right\}$.
    We are sometimes going to refer to $\d_j$ as $\e_{\overline{j}}$ for $1\leq j\leq n$.
\end{rem}

\paragraph{\textit{Highest weight representations}}The finite-dimensional simple $\mathfrak{b}$-highest weight representations of $\gl(m|n)$ are parametrized by pairs of partitions $(\lambda^1|\lambda^2)$ of lengths $l(\lambda^1)\leq m$ and $l(\lambda^2)\leq n$ and denoted by $V(\lambda^1|\lambda^2)$.

The natural representation of $\gl(m|n)$ is the superspace $\C^{m|n}=V(1^1|0^n)$ on which we have the weight basis
\begin{equation*}
    (\C^{m|n})_{\overline{0}}=\Big\langle e_i\ \Big|\ i\in[m]\Big\rangle_\C\hspace{1em} (\C^{m|n})_{\overline{1}}=\left\langle e_{\overline{j}}\ \Big|\ \overline{j}\in\overline{[n]}\right\rangle_\C.
\end{equation*}

The representations of interest in this article are the \textit{covariant representations} introduced by Berele-Regev \cite{Berele-Regev-Covariant} and Sergeev \cite{Sergeev-Covariant}.
\begin{defn}
    Let $\lambda,\mu$ be partitions of lengths $l(\lambda)\leq m$ and $l(\mu)\leq n$. Then, we say that the representation $V(\lambda|\mu)$ is \textit{covariant} if $\lambda_m\geq\left|\left\{ i\mid \mu_i\neq0\right\}\right|$.
\end{defn}
\begin{rem}
    Some pieces of the literature, like \cite{Cheng-Wang}, refer to these representations as \textit{polynomial representations}.
    \end{rem}
\begin{rem}
    These representations are exactly those that appear in the semisimple decomposition of ${(\mathbb{C}^{m|n})}^{\otimes d}$ for some $d\in\N$. The decomposition can be obtained using Schur-Weyl duality, see \cite[Theorem 3.11]{Cheng-Wang}.

    An alternative way of parameterizing covariant representations is by identifying them with a single $(m,n)$-hook Young diagram: We consider the composition $(\lambda,\mu')$, where $\mu'$ is the transpose of $\mu$. Then the covariance condition characterizes when this composition is a partition and hence its Young diagram lies in the $(m,n)$-hook.
\end{rem}

\begin{lem}\label{lem-cartan-embedding}
    Let $\lambda=(\lambda^1\mid\lambda^2)$, $\mu=(\mu^1\mid\mu^2)$ be pairs of partitions such that $V(\lambda)$ and $V(\mu)$ are covariant. Then $V(\lambda+\mu)$ is also covariant, and we have a (unique up to scaling) Cartan embedding
    \begin{align*}
        V(\lambda+\mu)&\longrightarrow V(\lambda)\otimes V(\mu)\\
        v_{\lambda+\mu}&\longmapsto v_\lambda\otimes v_\mu,
    \end{align*}
    where $v_{\lambda+\mu}$, $v_\lambda$ and $v_\mu$ are the respective highest weight vectors.
\end{lem}
\begin{proof}
    The vector on the right-hand side is of $\mathfrak{b}$-highest weight $\lambda+\mu$. We now need to check that $\mathcal{U}(\gl(m|n)).(v_\lambda\otimes v_\mu)$ is in fact irreducible. However, $V(\lambda)\otimes V(\mu)$ embeds in $V^{\otimes N}$ for an appropriate $N$ and is therefore semisimple. Thus the submodule generated by $v_\lambda\otimes v_\mu$ is in fact simple.
\end{proof}
We recall the PBW theorem for $\mathfrak n^-$:
\begin{thm}
Fix a homogeneous basis of $\mathfrak
n^-$ with $x_1,\dots,x_n$ a basis of $\mathfrak
n^-_{\overline{0}}$ and $y_1,\dots,y_q$ a basis of $\mathfrak
n^-_{\overline{1}}$, then the universal enveloping algebra $\mathcal{U}(\mathfrak{n}^-)$ has a basis consisting of ordered monomials
\begin{equation*}
    y_{i_{1}}\cdots y_{i_{l}}x_{1}^{k_{1}}\dots x_{n}^{k_{n}},\text{ for $k_i\geq 0$, $1\leq i_1<\cdots<i_l\leq q$}.
\end{equation*}
\end{thm}
If we consider the filtration of the universal enveloping algebra $\mathcal{U}(\mathfrak
n^-)$
\begin{equation*}
    \mathcal{U}(\mathfrak
n^-)=\bigcup_{k\in\N_0}\mathcal{U}(\mathfrak
n^-)_k,\hspace{1em}\mathcal{U}(\mathfrak
n^-)_k=\langle y_{i_{1}}\cdots y_{i_{l}}f_{1}^{k_{1}}\dots f_{n}^{k_{n}}\mid\hspace{0.5em}l+k_1+\cdots+k_n\leq k\rangle_\C
\end{equation*}

Then we get that the associated graded space $\gr\,\mathcal{U}(\mathfrak{n}^-)$ is isomorphic to $\Sym(\mathfrak n^-)$, the symmetric algebra of $\mathfrak n^-$.

This filtration induces a filtration on the simple highest weight representation $V(\lambda)$, where $\lambda=(\lambda^1|\lambda^2)$,
\begin{equation*}
    V(\lambda)=\bigcup\limits_{k\in\N_0}\mathcal{U}(\mathfrak{n}^{-})_kv_\lambda,
\end{equation*}
and $v_\lambda\in V(\lambda)$ denotes the highest weight vector.

On the associated graded space $\gr V(\lambda)$, we can then get an action of $\Sym(\mathfrak{n}^-)$ and thus have the isomorphism
\begin{equation*}
    \gr V(\lambda)\simeq \Sym(\mathfrak{n}^-)/I(\lambda)
\end{equation*}
for an appropriate left-ideal $I(\lambda)\subseteq \Sym(\mathfrak n^-)$.\\

\paragraph{\textit{Marked chain polytopes}}
Ardila, Bliem and Salazar showed in \cite{ABS} that the classical FFLV polytopes from \cite{FFLV_Type_A,FFLV_Type_C} are a generalized version of Stanley's chain polytope \cite{Stanley}.
\begin{defn}
     Let $\mathcal{P}$ be a poset and $A\subseteq \mathcal{P}$ be a subset with a vector $\lambda_{\mathcal{P}}\in\mathbb{R}^A$  such that $\lambda_{\mathcal{P}}(a)\leq\lambda_{\mathcal{P}}(b)$ whenever $a\leq b$. Then the \textit{marked chain polytope} of $(\mathcal{P},A,\lambda_{\mathcal{P}})$ is defined to be 
     \begin{equation*}
         \mathcal{C}(\mathcal{P},A)_{\lambda_{\mathcal{P}}}=\left\{x = (x_{p_i})\in\mathbb{R}^{\mathcal{P}\setminus A}\ \middle|\ \sum_{i=1}^{k}x_{p_i}\leq\lambda_{\mathcal{P}}(b)-\lambda_{\mathcal{P}}(a)\,\text{ for all }a<p_1<\cdots<p_k<b\right\},
     \end{equation*}
     where $a,b\in A$ and $p_1,\dots,p_k\in \mathcal{P}\setminus A$.
\end{defn}
 One family of our polytopes, those corresponding to the weights $(\lambda|0^n)$, will be a truncated version of these polytopes as we will see in Proposition \ref{prop-strip-fflv-is-marked-chain}.
\section{The Extended Dyck path model and polytope}\label{sec-defn-polytope}
We are now going to consider the following subset of $\tilde{\Phi}^+\subseteq\Phi^+$, where
\begin{align*}
    \tilde{\Phi}^+_{\overline{0}}&=\left\{\e_i-\e_j\,\middle|\,1\leq i< j\leq m,\right\}\\
    \tilde{\Phi}^+_{\overline{1}}&=\Phi^+_{\overline{1}}.
\end{align*}

On this subset of roots we want to extend the Dyck path model from \cite{FFLV_Type_A}.
\begin{defn}
    An \textit{extended Dyck path} is a sequence of roots in $\tilde{\Phi}^+$
    \begin{equation*}
        \mathbf{p}=(p(0),\dots,p(u))
    \end{equation*}
    satisfying either of the following:
    \begin{enumerate}[label=\alph*)]
        \item{
            The sequence $\mathbf{p}$ is a Dyck path involving the even roots $\e_i-\e_j\in\tilde{\Phi}^+_{\overline{0}}$.
        }
        \item{
            The following conditions hold:
            \begin{enumerate}[label=(\roman*)]
                \item{
                    There exists a $1\leq k\leq u+1$ such that $(p(k),\dots,p(u))$ is a Dyck path involving the even roots $\e_i-\e_j\in\tilde{\Phi}^+$ and $p(k)=\e_i-\e_m$ for some $i\in[m-1]$.
                }
                \item{
                    For all $0\leq i\leq k-1$ we have $p(i)\in\Phi^{+}_{\overline{1}}$. 
                }
                \item{
                    If $p(k)=\e_i-\e_m$, then $p(k-1)=\e_i-\d_1$.
                }
                \item{
                    For all $1\leq i\leq k-1$ if $p(i)=\e_a-\d_b$, then $p(i-1)=\e_c-\d_{b+1}$ with $a\leq c$.
                }
            \end{enumerate}
        }
    \end{enumerate}
\end{defn}
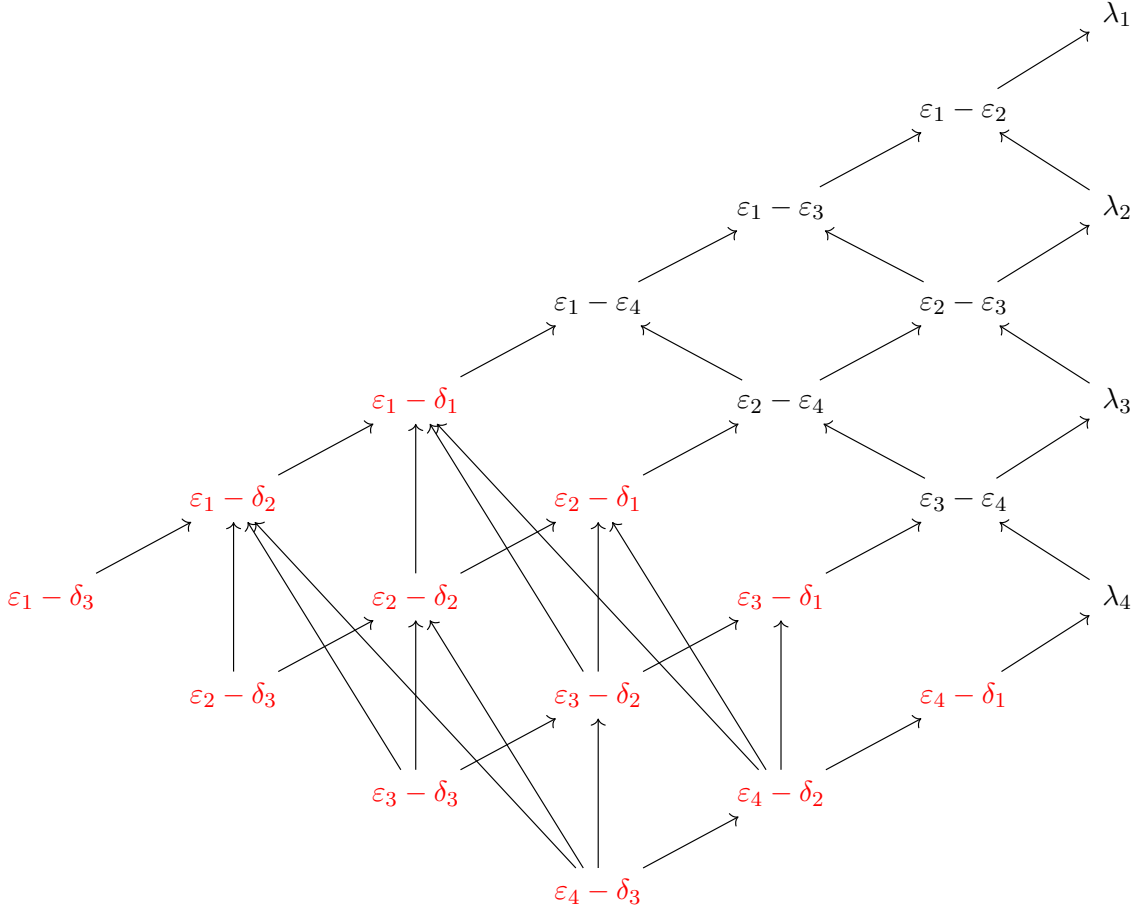
\begin{figure}
\centering
$\begin{tikzcd}
	&&&&&& {\lambda_1} \\
	&&&&& {\varepsilon_1-\varepsilon_2} \\
	&&&& {\varepsilon_1-\varepsilon_3} && {\lambda_2} \\
	&&& {\varepsilon_1 - \varepsilon_4} && {\varepsilon_2-\varepsilon_3} \\
	&& {\textcolor{red}{\varepsilon_1-\delta_1}} && {\varepsilon_2-\varepsilon_4} && {\lambda_3} \\
	& {\textcolor{red}{\varepsilon_1 - \delta_2}} && {\textcolor{red}{\varepsilon_2-\delta_1}} && {\varepsilon_3-\varepsilon_4} \\
	{\textcolor{red}{\varepsilon_1 - \delta_3}} && {\textcolor{red}{\varepsilon_2-\delta_2}} && {\textcolor{red}{\varepsilon_3 - \delta_1}} && {\lambda_4} \\
	& {\textcolor{red}{\varepsilon_2 - \delta_3}} && {\textcolor{red}{\varepsilon_3-\delta_2}} && {\textcolor{red}{\varepsilon_4-\delta_1}} \\
	&& {\textcolor{red}{\varepsilon_3 - \delta_3}} && {\textcolor{red}{\varepsilon_4-\delta_2}} \\
	&&& {\textcolor{red}{\varepsilon_4 - \delta_3}}
	\arrow[from=2-6, to=1-7]
	\arrow[from=3-5, to=2-6]
	\arrow[from=3-7, to=2-6]
	\arrow[from=4-4, to=3-5]
	\arrow[from=4-6, to=3-5]
	\arrow[from=4-6, to=3-7]
	\arrow[from=5-3, to=4-4]
	\arrow[from=5-5, to=4-4]
	\arrow[from=5-5, to=4-6]
	\arrow[from=5-7, to=4-6]
	\arrow[from=6-2, to=5-3]
	\arrow[from=6-4, to=5-5]
	\arrow[from=6-6, to=5-5]
	\arrow[from=6-6, to=5-7]
	\arrow[from=7-1, to=6-2]
	\arrow[from=7-3, to=5-3]
	\arrow[from=7-3, to=6-4]
	\arrow[from=7-5, to=6-6]
	\arrow[from=7-7, to=6-6]
	\arrow[from=8-2, to=6-2]
	\arrow[from=8-2, to=7-3]
	\arrow[from=8-4, to=5-3]
	\arrow[from=8-4, to=6-4]
	\arrow[from=8-4, to=7-5]
	\arrow[from=8-6, to=7-7]
	\arrow[from=9-3, to=6-2]
	\arrow[from=9-3, to=7-3]
	\arrow[from=9-3, to=8-4]
	\arrow[from=9-5, to=5-3]
	\arrow[from=9-5, to=6-4]
	\arrow[from=9-5, to=7-5]
	\arrow[from=9-5, to=8-6]
	\arrow[from=10-4, to=6-2]
	\arrow[from=10-4, to=7-3]
	\arrow[from=10-4, to=8-4]
	\arrow[from=10-4, to=9-5]
\end{tikzcd}
$
\caption{Extended Dyck path pattern for $\gl(4|3)$-weights $(\lambda_1,\lambda_2,\lambda_3,\lambda_4,|0,0,0)$. The odd roots are marked in red to indicate that the corresponding $s_\alpha$ are bounded by $1$.}
\label{fig-gl42-strip-pattern}
\end{figure}
\begin{defn}
    Let $\lambda$ be a partition of length $l(\lambda)\leq m$. We define the polytope $P(\lambda|0^n)\subseteq\mathbb{R}^{\tilde{\Phi}^+}_{\geq0}$ to consist of all points $(s_{\alpha}\ |\ \alpha\in \tilde\Phi^+)$ satisfying the following conditions for every extended Dyck path $\mathbf{p}=(p(0),\dots,p(u))$:
    \begin{align}
            \sum_{t=0}^{u}s_{p(t)}&\leq\lambda_i-\lambda_j\hspace{1em}&&\text{ if }p(0)=\e_{j-1}-\e_{j}\text{ and }p(u)=\e_i-\e_{i+1}\label{eq-inequality-fflv-strip-even-dyck-path}\\
            \sum_{t=0}^{u}s_{p(t)}&\leq\lambda_r&&\text{ if }p(0)\in\Phi^+_{\overline{1}}\text{ and }p(u)=\e_r-\e_{r+1}\label{eq-inequality-fflv-strip-extended-dyck-path}\\
            s_{\e_{i}-\d_{{j}}}&\leq 1&&\text{ for all }(i,j)\in[m]\times[n]\label{eq-inequality-fflv-strip-odd-root-bound-1}
    \end{align}
    We shorten the notation for the coordintes by setting $s_{i,j}:=s_{\e_i-\e_j}$ for $\e_i-\e_j\in\tilde{\Phi}_{\overline{0}}^+$.
    
    We denote the set of lattice points of $P(\lambda|0^n)$ by $S(\lambda|0^n):=P(\lambda|0^n)\cap\Z^{\tilde{\Phi}^+}$.
\end{defn}
\begin{rem}\label{rem-strip-typical-matches-fourier-kus}
    If $\lambda_m\geq n$, which is exactly the case when $V(\lambda|0^n)$ is typical, then the conditions \eqref{eq-inequality-fflv-strip-extended-dyck-path} become redundant and thus $P(\lambda|0^n)$ matches the polytope from \cite[Theorem 1.1]{Fourier-Kus}
\end{rem}
\begin{prop}\label{prop-strip-fflv-is-marked-chain}
    Let $\lambda$ be a partition of length $l(\lambda)\leq m$. Then there exists a marked chain polytope that we call $\hat{P}(\lambda)\subseteq\mathbb{R}^{\tilde{\Phi}^+}_{\geq0}$ such that
    \begin{equation*}
        P(\lambda|0^n)=\hat{P}(\lambda)\cap\left\{\mathbf{s}\in\mathbb{R}^{\tilde{\Phi}^+}\,\big|\, s_\alpha\in\left[0,1\right]\text{ for all }\alpha\in\Phi^+_{\overline1
        }\right\}.    \end{equation*}
\end{prop}
\begin{proof}
    We construct the following marked poset $(\mathcal{P},A,\lambda_{\mathcal{P}})$:
    Define the sets
    \begin{equation*}
        \mathcal{P}:=\left\{x_{i,j}\mid i\in[m]\,,\,1\leq j\leq m+1-i\right\}\cup\left\{x_{i,\overline{j}}\mid (i,j)\in[m]\times\overline{[n]}\right\}\cup\left\{\hat{0},\hat{1}\right\},
    \end{equation*}
    and
    \begin{equation*}
        A=\left\{x_{1,j}\mid j\in[m]\right\}\cup\left\{\hat{0},\hat{1}\right\}.
    \end{equation*}
    We set the covering relations
    \begin{equation}
        \begin{aligned}[t]
            x_{i-1,j+1}\geq  &\ x_{i,j}\geq x_{i-1,j}\\
            x_{i,1}\geq &\ x_{i,\overline{1}}\\
            x_{i,\overline{j}}\geq &\ x_{k,\overline{j+1}}\\
             &\ x_{i,\overline{j}}\geq\hat{0}
        \end{aligned}
        \qquad 
        \begin{aligned}[t]
            &{}\\
            &\text{ for all }i\in[m]\\
            &\text{ for all }k\leq i\\
            &\text{ for all }(i,\overline{j})\in[m]\times\overline{[n]},
        \end{aligned}
    \end{equation}
    We always define the first line of relations when the set of indices are defined.
    We mark $x_{1,j}$ by $\lambda_j$, and we mark $\hat{0}$ by $0$ and $\hat{1}$ by 1. Let $\lambda_\mathcal{P}$ be the corresponding marking.
    Then we have that $\mathcal{C}(\mathcal{P},A)_{\lambda_{\mathcal{P}}}=\hat{P}(\lambda)$. Only considering those points $\mathbf{s}\in\hat{P}(\lambda)$ with $s_{\alpha}\leq 1$ for all $\alpha\in\Phi^{+}_{\overline{1}}$ gives us exactly $P(\lambda|0^n)$.
\end{proof}
\begin{rem}
    In the case of partitions $\lambda=(1^k)$ for $1\leq k\leq m$, we have that $P(\lambda|0^n)$ is a genuine chain polytope. This will be important in section \ref{sec-indep-column-strip}.
\end{rem}
\begin{defn}
    We define for a multi-exponent $\mathbf
    {s}\in\Z^{\tilde{\Phi}^+}$ the monomial
    \begin{equation*}
        f^\mathbf{s}:=\prod_{\alpha\in\tilde{\Phi}^+}f^{s_{\alpha}}_{\alpha}\in \Sym(\mathfrak{n}^{-})
    \end{equation*}
\end{defn}
Now, we can fully state the first result.
\begin{thm}\label{thm-fflv-strip}
    Let $\lambda$ be a partition of length $l(\lambda)\leq m$. Then the set
    \begin{equation*}
        \left\{f^{\mathbf{s}}\cdot v_{(\lambda|0^n)}\ \middle|\ \mathbf{s}\in S(\lambda|0^n)\right\}
    \end{equation*}
    provides a basis of the associated graded space $\gr V(\lambda|0^{n})$.
\end{thm}
We are going to prove this theorem as follows:
\begin{enumerate}
    \item{
        In section \ref{sec-spanning-strip}, we show that our set of monomials is a spanning set of $\gr V(\lambda|0^n)$.
    }
    \item{
        In section \ref{sec-indep-column-strip}, we prove that our set of monomials is linearly independent in $\gr V(\lambda|0^n)$ for $\lambda=(1^k),\ 1\le k\le m$, i.e., for $V(\lambda|0^n) \simeq \bigwedge^k\C^{m|n}$.
    }
    \item{
        In section \ref{sec-indep-strip}, we argue how the previous step yields us the linear independence for general $\lambda$ by proving a Minkowski decomposition of our polytopes.
    }
\end{enumerate}
\section{Spanning property for the strip $(\lambda|0^n)$}\label{sec-spanning-strip}

The goal of this section is to prove that the set of vectors $\{f^{\mathbf{s}}\cdot v_\lambda\ |\ \textbf{s}\in S(\lambda|0^n)\}$ is in fact a spanning set of $\mathrm{gr}\,V(\lambda|0^n)$. Before we go into proving this, we believe that it is instructive to recall how Feigin, Fourier and Littelmann proved that their set of vectors was spanning in \cite{FFLV_Type_A}:

For any dominant integral $\mathfrak{sl}_{n+1}$-weight $\lambda$, they show that any monomial $f^{\mathbf{s}}$ with $\mathbf{s}\not\in S(\lambda)$ can be generated by the monomials corresponding to elements of $S(\lambda)$. To achieve this, they proceeded as follows:
\begin{enumerate}
    \item{
        Fix a monomial well-ordering on the set of monomials $f^{\mathbf{m}}\in \Sym({\mathfrak{n}^{-}_{\mathfrak{sl}_{n+1}}})$.
    }
    \item{
        Find appropiate monomials inside the ideal $f^{\mathbf{s'}}\in I(\lambda)$.
    }
    \item{
        Apply a sequence of derivations $\partial_\alpha$, with $\alpha\in\Phi_{\mathfrak{sl}_{n+1}}^+$, to $f^{\mathbf{s'}}$ such that we arrive at a linear combination, i.e. a straightening law,
        \begin{equation*}
            f^{\mathbf{s}}+\sum_{\mathbf{s}\succ\mathbf{t}}c_{\mathbf{t}}f^{\mathbf{t}}\in I(\lambda).
        \end{equation*}
        This allows one to then use the well-ordering $\succ$ to get that the set is a spanning set.
    }
\end{enumerate}

\vspace{4mm}
Our proof will follow a very similar recipe. We begin by imposing the following monomial order on $\Sym(\mathfrak n^-)$:
\begin{defn}\label{def-monomial-order}
    We define the \textit{monomial order} $<$ on $\Sym({\mathfrak{n}}^-)$ to be the homogenous lexicographic order where we order the variables
    \begin{multline*}
        f_{\d_{m-1}-\d_m}> f_{\d_{m-2}-\d_m}> f_{\d_{m-2}-\d_{m-1}}> f_{\d_{m-3}-\d_m}>\ldots >f_{\d_{1}-\d_m}>\ldots> f_{\d_{1}-\d_2}\\
        >f_{\e_{m}-\d_n}>f_{\e_m-\d_{n-1}}>\cdots>f_{\e_m-\d_1}>f_{\e_{m-1}-\d_{n}}>\cdots f_{\e_{m-1}-\d_{1}}>f_{\e_{m-1}-\e_{m}}>\\f_{\e_{m-2}-\d_n}>\cdots f_{\e_{m-2}-\d_1}>f_{\e_{m-2}-\e_m}>f_{\e_{m-2}-\e_{m-1}}>\cdots >f_{\e_{1}-\e_{2}}
    \end{multline*}
    Intuitively speaking, we go along the (bottom left to top right) diagonals from bottom to top in patterns of the form Figure \ref{fig-gl42-strip-pattern} (when no $f_{\d_i-\d_j}$ variables are involved) or \eqref{eq-fflv-pattern-hook-3|2}.
\end{defn}
Next, we need to find an appropriate monomial inside the ideal $I(\lambda|0^n)\subseteq \Sym({\mathfrak{n}}^-)$ to which we will apply the derivations.
\begin{lem}\label{lem-monomial-in-ideal}
    For any $1\leq i\leq m$ and $I\subseteq{[n]}$ such that $|I|=k$, we have
    \begin{equation}\label{eq-monomial-in-ideal}
        f^{\lambda_{i}+1-k}_{\e_i-\e_m}\prod_{j\in I}f_{\e_i-\delta_j}\in I(\lambda|0^n).
    \end{equation}
\end{lem}
\begin{proof}
    We consider the highest weight vector of $V(\lambda|0^{n})$ via the Cartan embedding from Lemma \ref{lem-cartan-embedding}
    \begin{equation*}
        v_\lambda=(e_1)^{\otimes(\lambda_1-\lambda_2)}\otimes(e_1\wedge e_2)^{\otimes(\lambda_2-\lambda_3)}\otimes\cdots\otimes(e_1\wedge e_2\wedge\cdots\wedge e_m)^{\otimes(\lambda_{m-1}-\lambda_{m})}.
    \end{equation*}
    Then each factor from \eqref{eq-monomial-in-ideal} can only act on tensor factors involving the basis vector $e_i$ of which there are exactly $\lambda_i$ many in $v_\lambda$. Hence, this monomial acts as $0$ on $v_\lambda$ and is therefore in $I(\lambda|0^n)$.
\end{proof}
\begin{rem}
    As opposed to the classical case, the left ideal $I(\lambda|0^n)\subset\mathcal{U}(\gl(m|n))$ annihilating the highest weight vector $v_{(\lambda|0^n)}\in V(\lambda|0^n)$ is not explicitly described via generators. Hence, we are forced to proceed in this more elementary manner using the Cartan component embedding to find appropriate monomials in $I(\lambda|0^n)$.
\end{rem}
Finally, we introduce the set of differential operators, which when applied to our monomial in $I(\lambda|0^n)$ give the desired straightening law.

We recall that on \cite[p. 78]{FFLV_Type_A}, Feigin, Fourier and Littelmann defined for a positive $\mathfrak{sl}_{n+1}$-root $\alpha\in\Phi_{\mathfrak{sl}_{n+1}}^+$ a differential operator $\partial_{\alpha}:\Sym({\mathfrak{n}^{-}_{\mathfrak{sl}_{n+1}}})\to \Sym({\mathfrak{n}^{-}_{\mathfrak{sl}_{n+1}}})$ that behaves like the adjoint action by $e_{\alpha}$:
\begin{equation*}
    \partial_{\alpha}f_\beta=\begin{cases}
        f_{\beta-\alpha},\hspace{1em}&\text{ if }\beta-\alpha\in\Phi_{\mathfrak{sl}_{n+1}}^+\\
        0,\hspace{1em}&\text{ otherweise}
    \end{cases}
\end{equation*}
We define our set of operators similarly, as follows
\begin{defn}
    For a positive root $\alpha\in\Phi^+$, we define the \textit{differential operator $\partial_{\alpha}$} on $\Sym(\mathfrak{n}^-)$ acting on a root vector $f_{\beta}$ via
    \begin{equation*}
        \begin{aligned}
            \partial_{\alpha}f_{\beta}=\begin{cases}
            f_{\beta-\alpha},\hspace{1em}&\text{ if $\beta-\alpha\in\Phi^+$}\\
            0,\hspace{1em}&\text{ otherwise}
        \end{cases},
        \end{aligned}
    \end{equation*}
\end{defn}
\begin{rem}
    If $\alpha=\e_a-\e_b\in\tilde
    \Phi^+_{\overline{0}}$, then we will denote $\partial_{\alpha}$ as $\partial_{a,b}$.

    Furthermore, if $\beta=\e_i-\d_j\in\Phi^+_{\overline{1}}$ and $\alpha=\e_a-\e_b\in\tilde{\Phi}^+_{\overline{0}}$, this action can be simplified to
    \begin{equation*}
        \partial_{a,b}f_{\e_i-\d_j}=\delta_{i,a}\cdot f_{\e_b-\d_j}.
    \end{equation*}
\end{rem}
Now, we can show that our set of vectors is a  generating set by proving the straightening law. We only need to consider a violation to the extended Dyck path inequality \eqref{eq-inequality-fflv-strip-extended-dyck-path} corresponding to an extended Dyck path beginning in an odd root, since for purely even Dyck paths we can apply \cite[Proposition 1]{FFLV_Type_A}.
\begin{prop}\label{prop-straightening}
    Let $\mathbf{p}=(p(0),\dots,p(u))$ be an extended Dyck path with $p(0)=\e_a-\d_b$ and $p(u)=\e_r-\e_{r+1}$.
    Let $\mathbf{s}$ be an multi-exponent supported on $\mathbf{p}$. Assume further that
    \begin{equation*}
        \sum_{t=0}^{u}s_{p(t)}>\lambda_r.
    \end{equation*}
    Then there exist some constants $c_{\mathbf{t}}$ labeled by multi-exponents $\mathbf{t}$ such that
    \begin{equation*}
        f^{\mathbf{s}}+\sum_{\mathbf{s}\succ\mathbf{t}}c_{\mathbf{t}}f^{\mathbf{t}}\in I(\lambda|0^n)
    \end{equation*}
\end{prop}
\begin{proof}
    We can restrict ourselves to the case of a violation of the inequalities involing extended Dyck paths ending in $\e_1-\e_2$ and the inequality being violated by $1$, i.e.
    \begin{equation*}
        \sum_{t=0}^{u}s_{p(t)}=\lambda_1+1.
    \end{equation*}
    
    We define the sets
    \begin{equation*}
        \mathrm{OR}_i(\mathbf{s})=\{j\in[n]\mid s_{\e_i-\d_j}=1\},
    \end{equation*}
    as well as
    \begin{equation*}
        I:=\bigcup_{1\leq i\leq m}\mathrm{OR}_i(\mathbf{s})
    \end{equation*}
    and fix $k:=|I|$.
    
    Intuitively speaking, the set $I$ is keeping track of all horizontal levels of odd roots where $\mathbf{s}$ has non-empty support and the sets $\mathrm{OR}_i(\mathbf{s})$ are then the appropriate vertical level sets.
    
    We consider now the monomial from Lemma \ref{lem-monomial-in-ideal} 
    \begin{equation*}
        f^{\lambda_{1}+1-k}_{\e_1-\e_m}\prod_{j\in I}f_{\e_1-\delta_j}.
    \end{equation*}
    We define the following quantities for $1\leq i,j\leq m-1$
    \begin{equation*}
        s_{\bullet,j}=\sum_{i=1}^{j}s_{i,j},\hspace{1em}s_{i,\bullet}=\sum_{j=i}^{m-1}s_{i,j}.
    \end{equation*}
    We define the set
     \begin{equation*}
        \tau=\{i\in[m]\mid\mathrm{OR}_i(\mathbf{s})\neq\emptyset\},
    \end{equation*}
    and note that in fact
    \begin{equation*}
        s_{i,\bullet}=0\text{ for }i>\min\tau.
    \end{equation*}
    We consider first the vector
    \begin{equation*}
        y:=\partial^{s_{\bullet,m-2}}_{m-1,m}\cdots\partial^{s_{\bullet,1}}_{2,m}\left(f^{\lambda_{1}+1-k}_{\e_1-\e_m}\prod_{j\in I}f_{\e_1-\delta_j}\right)\in I(\lambda|0^n),
    \end{equation*}
    which actually
    \begin{equation*}
        y=f_{\e_1-\e_2}^{s_{\bullet,1}}f^{s_{\bullet,2}}_{\e_1-\e_3}\cdots f_{\e_1-\e_m}^{s_{\bullet,m-1}}\prod_{j\in I}f_{\e_1-\delta_j}\in I(\lambda|0^n).
    \end{equation*}
    Before we state our next sequence of derivations, we want to introduce one more piece of notation. 
    Namely, we define the new quantities
    \begin{equation*}
        \widetilde{s_{i,\bullet}}:=s_{i,\bullet}+|\mathrm{OR}_i(\mathbf{s})|.
    \end{equation*}
    We now want to apply the following sequence of derivations to $y$
    \begin{equation*}
        A=\partial_{1,2}^{\widetilde{s_{2,\bullet}}}\partial_{1,3}^{\widetilde{s_{3},\bullet}}\cdots\partial^{\widetilde{s_{a,\bullet}}}_{1,a}\left(f_{\e_1-\e_2}^{s_{\bullet,1}}f^{s_{\bullet,2}}_{\e_1-\e_3}\cdots f_{\e_1-\e_m}^{s_{\bullet,m-1}}\prod_{j\in I}f_{\e_1-\delta_j}\right)\in I(\lambda|0^n),
    \end{equation*}
    where $a=\max\tau$.
    This element $A$ will allow us to arrive at our conclusion by proving the following claim.
    \begin{claim}
        We have that
            \begin{equation*}
                A=f^{\mathbf{s}}+\sum_{\mathbf{s}\succ\mathbf{t}}c_{\mathbf{t}}f^{\mathbf{t}}\in I(\lambda|0^n),
            \end{equation*}
        for some constants $c_{\mathbf{t}}\in\C$.
    \end{claim}
    \begin{proof}[Proof of Claim]
    \renewcommand{\qedsymbol}{}
    We are going to prove this by induction on the number $q:=|\tau|$.
    
    In the case $q=1$, we have that $\mathbf{p}$ is of the form
    \begin{equation*}
        \mathbf{p}=(\e_a-\d_{b_1},\e_a-\d_{b_2},\dots,\e_a-\d_{b_k},\text{ even roots}).
    \end{equation*}
    And our element $A$ can be written as
    \begin{equation*}
        A=\partial_{1,2}^{s_{2,\bullet}}\partial_{1,3}^{s_{3},\bullet}\cdots\partial^{s_{a,\bullet}+k}_{1,a}\left(f_{\e_1-\e_2}^{s_{\bullet,1}}f^{s_{\bullet,2}}_{\e_1-\e_3}\cdots f_{\e_1-\e_m}^{s_{\bullet,m-1}}\prod_{j\in I}f_{\e_1-\delta_j}\right).
    \end{equation*}
    Expanding the derivation $\partial^{k}_{1,a}$ on $y$ yields us
    \begin{equation}\label{eq-straightening-inductio-base}
        \sum_{J\subseteq I}c_J\partial^{k-|J|}_{1,a}\left(f_{\e_1-\e_2}^{s_{\bullet,1}}f^{s_{\bullet,2}}_{\e_1-\e_3}\cdots f_{\e_1-\e_m}^{s_{\bullet,m-1}}\right)\prod_{j\in J}f_{\e_a-\delta_j}\prod_{j\not\in J}f_{\e_1-\delta_j},
    \end{equation}
    for some non-zero constants $c_J\neq0$.
    Now note that applying $\partial^{k-|J|}_{1,a}$ will only result in a linear combination of monomials whose 
    variables are bounded by $f_{\e_a-\e_m}$ in the monomial order $<$.
    
    We can see this by remembering that for any positive even root $\e_c-\e_d\in\Phi^+_{\overline{0}}$, we have that $\e_c-\e_d-\e_1-\e_g$ is only a positive root in $\Phi^+_{\overline{0}}$ if it is equal to $\e_g-\e_d$.

    What this means in particular is that the leading term of \eqref{eq-straightening-inductio-base} is the leading term of 
    \begin{equation*}
        \left(f_{\e_1-\e_2}^{s_{\bullet,1}}f^{s_{\bullet,2}}_{\e_1-\e_3}\cdots f_{\e_1-\e_m}^{s_{\bullet,m-1}}\right)\prod_{j\in I}f_{\e_a-\delta_j}.
    \end{equation*}
    Now with the same argument we can conclude that the leading term of $A$ is then the leading term of
    \begin{equation*}
        \partial_{1,2}^{s_{2,\bullet}}\partial_{1,3}^{s_{3},\bullet}\cdots\partial^{s_{a,\bullet}}_{1,a}\left(f_{\e_1-\e_2}^{s_{\bullet,1}}f^{s_{\bullet,2}}_{\e_1-\e_3}\cdots f_{\e_1-\e_m}^{s_{\bullet,m-1}}\right)\prod_{j\in I}f_{\e_a-\delta_j}.
    \end{equation*}
    Using the arguments from \cite[Proposition 1]{FFLV_Type_A} we can thus conclude that the leading term of $A$ is then our desired monomial.

    Now for the induction step, we let
    \begin{equation*}
        \tau=\{t_1,\dots,t_{q}=a\}.
    \end{equation*}
    Hence, our element $A$ is now of the form
    \begin{equation*}
        A=\partial_{1,2}^{s_{2,\bullet}}\partial_{1,3}^{s_{3},\bullet}\cdots\partial^{s_{t_{1}-1,\bullet}}_{1,t_{1}-1}\partial^{s_{t_{1},\bullet}+k_{t_{1}}}_{1,t_{1}}\partial^{k_{t_{2}}}_{1,t_{2}}\cdots\partial^{k_{t_{q}-1}}_{1,t_{q}-1}\partial^{k_{t_{q}}}_{1,a}\left(f_{\e_1-\e_2}^{s_{\bullet,1}}f^{s_{\bullet,2}}_{\e_1-\e_3}\cdots f_{\e_1-\e_m}^{s_{\bullet,m-1}}\prod_{j\in I}f_{\e_1-\delta_j}\right).
    \end{equation*}
    Now distributing $\partial_{1,a}^{k_{q}}$ over $y$ yields us again
    \begin{equation*}
        \resizebox{\textwidth}{!}{$\displaystyle
        \sum_{\substack{J\subseteq I\\|J|\leq k_{q}}}c_J\partial_{1,2}^{s_{2,\bullet}}\partial_{1,3}^{s_{3},\bullet}\cdots\partial^{s_{t_{1}-1,\bullet}}_{1,t_{1}-1}\partial^{s_{t_{1},\bullet}+k_{t_{1}}}_{1,t_{1}}\partial^{k_{t_{2}}}_{1,t_{2}}\cdots\partial^{k_{t_{q}-1}}_{1,t_{q}-1}\left(\partial^{k_{t_{q}}-|J|}_{1,a}\left(f_{\e_1-\e_2}^{s_{\bullet,1}}f^{s_{\bullet,2}}_{\e_1-\e_3}\cdots f_{\e_1-\e_m}^{s_{\bullet,m-1}}\right)\prod_{j\not\in I\setminus J}f_{\e_1-\d_j}\right)\prod_{j\in J}f_{\e_a-\delta_j},
        $}
    \end{equation*}
    for some non-zero constants $c_J$.
    
    Now, we can argue as in the base case and see that the leading term of $A$ is the leading term of
    \begin{equation*}
        \resizebox{\textwidth}{!}{$\displaystyle
        \partial_{1,2}^{s_{2,\bullet}}\partial_{1,3}^{s_{3},\bullet}\cdots\partial^{s_{t_{1}-1,\bullet}}_{1,t_{1}-1}\partial^{s_{t_{1},\bullet}+k_{t_{1}}}_{1,t_{1}}\partial^{k_{t_{2}}}_{1,t_{2}}\cdots\partial^{k_{t_{q}-1}}_{1,t_{q}-1}\left(f_{\e_1-\e_2}^{s_{\bullet,1}}f^{s_{\bullet,2}}_{\e_1-\e_3}\cdots f_{\e_1-\e_m}^{s_{\bullet,m-1}}\prod_{j\not\in I\setminus \mathrm{OR}_a(\mathbf{s})}f_{\e_1-\d_j}\right)\prod_{j\in \mathrm{OR}_a(\mathbf{s})}f_{\e_a-\delta_j},
        $}
    \end{equation*}
    The reason why we can get the set $\mathrm{OR}_{a}(\mathbf{s})$ is that $\mathrm{OR}_{a}(\mathbf{s})$ consists of the $k_q$ many largest elements of $I$.

    Applying our induction hypothesis proves the claim.
    \end{proof}
    \end{proof}
\section{Bijection for small exterior powers}\label{sec-indep-column-strip}

Before proving that our set of monomials yields us a basis for $\gr V(1^k|0^n)$, we introduce the following objects:

We recall from Remark \ref{rem-barred-indices-and-defn-set-I} the set $K = [m]\cup \overline{[n]}$ on which we introduce the total order 
\[1<\cdots < m< \overline 1 < \cdots < \overline n.\]
\begin{defn}
Fix $1\le k\le m$. We define $B_k$ to be the set of functions $f:[k]\to K$ satisfying the following conditions:
\begin{itemize}
    \item If $f(i)\in [k]$, then $f(i) = i$, and
    \item For $i< j$ such that $\{f(i), f(j)\}\subset K\setminus [k]$, either $f(i)>f(j)$ or $f(i)=f(j)\in \overline{[n]}$.
\end{itemize}
\end{defn}
\begin{lem}
    $B_k$ indexes a weight basis of $\bigwedge^k\C^{m|n}$ under the mapping 
\[f\mapsto e_{f(1)}\wedge \cdots \wedge e_{f(k)}.\]
\end{lem}

\begin{proof}
    To see the inverse bijection, any weight vector is of the form $e_{\underline a}:= e_{a_1}\wedge \cdots \wedge e_{a_k}$, where $\underline a = (a_1, \dots, a_k)\in K^k$ is a tuple such that $a_i = a_j$ only if they both belong to $\overline{[n]}$. Then there is a unique way to reorder the tuple $\underline a$ such that 
\begin{itemize}
    \item If $a_i\in [k]$, then $a_i = i$, and
    \item the $a_i$ such that $a_i\in K\setminus[k]$ are arranged in (weakly) decreasing order.
\end{itemize}
Then $f(i) = a_i$ gives the desired $f\in B_k$.
\end{proof}

Now we analyze the lattice points of the marked chain polytope $S(\lambda|0^n)$ with $\lambda = \omega_k = (1^k)$, i.e., $\lambda_1 = \cdots = \lambda_k = 1, \lambda_{k+1}=\cdots = \lambda_m = 0$, and show that these are also in bijection with $B_k$.

\begin{lem}\label{lem-bijection-small-fundamentals}
    The lattice points $S(\omega_k|0^n)$ of $P(\omega_k|0^n)$ are in bijection with $B_k$
\end{lem}

\begin{proof}
    The chain condition immediately implies that for $i<j\in K$, $0\le s_{\e_i - \e_j} \le 1$ (recall that $\e_{\overline t}:= \d_t$ for $t\in [n]$). Moreover if $j\le k$ or $i>k$, we have $s_{\e_i-\e_j} = 0$. Thus the only interesting coordinates are $s_{\e_i-\e_j}$ with $1\le i\le k < j\le \overline n$. For each $i\le k$, the (bottom left to top right) diagonal chain shows that there exists at most one $j$ with $s_{\e_i-\e_j} = 1$. 
    
    If $1\le i_1<i_2\le k < j_1 \le j_2\le \overline n$, there is a chain showing that $s_{\e_{i_1}-\e_{j_1}} + s_{\e_{i_2}-\e_{i_2}}\le 1$. In other words, at most one of $s_{\e_{i_1}-\e_{j_1}}$ and $s_{\e_{i_2}-\e_{j_2}}$ can be nonzero. Moreover, if $j_1=j_2\notin \overline{[n]}$, a similar statement holds. 
    
    Hence every point in $S(\lambda)$ can be described by first choosing:
    \begin{itemize}
        \item $1\le i_1<\cdots < i_t\le k$, and 
        \item $\overline n\ge  j_1\ge \cdots \ge j_t> k$, with $j_p = j_q\implies j_p\in \overline {[n]}$,
    \end{itemize}
    and then declaring $s_{\e_{i_1}-\e_{j_{1}}} = \cdots = s_{\e_{i_1}-\e_{j_{1}}} = 1$, and all other coordinates are set to $0$. This uniquely defines a function $f\in B_k$ by the formula:
    \begin{equation*}
        f(i) = \begin{cases}
            j_p & \text{if }i=i_p\\
            i & \text{if }i\notin \{i_1, \dots, i_t\}
        \end{cases}
    \end{equation*}
    It is easy to check that this is a bijection.
\end{proof}

\section{Linear independence for arbitrary weights $(\lambda|0^n)$ in the strip}\label{sec-indep-strip}
We now want to show that our set of monomials is in fact linearly independent in $\gr V(\lambda|0^n)$ for any partition $\lambda$ of length $l(\lambda)\leq m$.

First, we restate the Minkowski sum property for the chain polytopes $\hat{P}(\lambda)$.
\begin{prop}\label{prop-minkowski-decomp-chain}
    Let $\lambda,\mu$ be partitions of lengths $l(\lambda),l(\mu)\leq m$. Let $\hat{S}(\lambda)=\hat{P}(\lambda)\cap\Z^{\tilde{\Phi}^{+}}$, $\hat{S}(\mu)=\hat{P}(\mu)\cap\Z^{\tilde{\Phi}^{+}}$. 
    Then we have that
    \begin{equation*}
        \hat{S}(\lambda+\mu|0^n)=\hat{S}(\lambda|0^n)+\hat{S}(\mu|0^n)
    \end{equation*}
\end{prop}
\begin{proof}
    We recall from \cite[Lemma 2]{Fourier-Chain-Polytopes} that we have the decomposition
    \begin{equation*}
        S(\lambda|0^n)=S_{\mathcal{C}}(\lambda_\mathcal{P})=S_\mathcal{C}(\lambda_\mathcal{P}-\omega_\mathcal{P})+S_\mathcal{C}(\omega_\mathcal{P}),
    \end{equation*}
    where $\omega_\mathcal{P}$ is the marking
    \begin{equation*}
        \omega_\mathcal{P}(a):=\begin{cases}
            1\text{ if }\lambda_\mathcal{P}(a)\neq0\\
            0\text{ otherwise}
        \end{cases}.
    \end{equation*}
    Note that $S_{\mathcal{C}}(\omega_\mathcal{P})=\hat{S}(1^{l(\lambda)})$ and $S_{\mathcal{C}}(\lambda-\omega_\mathcal{P})=\hat{S}(\lambda-1^{l(\lambda)})$.
    Inducting down along $l(\lambda)$, we see that
    \begin{equation*}
        \hat{S}(\lambda)=\sum_{i=1}^{l(\lambda)}(\lambda_i-\lambda_{i+
        1})\hat{S}(1^i),
    \end{equation*}
    where $\lambda_{l(\lambda)+1}=0$. In particular, we have
    \begin{align*}
        \hat{S}(\lambda)+\hat{S}(\mu)&=\left(\sum_{i=1}^{l(\lambda)}(\lambda_i-\lambda_{i+
        1})\hat{S}(1^i)\right)+\left(\sum_{i=1}^{l(\mu)}(\mu_i-\mu_{i+
        1})\hat{S}(1^i)\right)\\
        &=\sum_{i=1}^{l(\lambda+\mu)}((\lambda+\mu)_i-(\lambda+\mu)_{i+
        1})\hat{S}(1^i)=\hat{S}(\lambda+\mu)
    \end{align*}
\end{proof}
\begin{prop}\label{prop-minkowski-decomp}
    Let $\lambda,\mu$ be partitions of lengths $l(\lambda),l(\mu)\leq m$. Then we have that
    \begin{equation*}
        S(\lambda+\mu|0^n)=(S(\lambda|0^n)+S(\mu|0^n))\cap\left\{\mathbf{m}\in\Z^{\tilde{\Phi}^+}\,\big|\, m_\alpha\in\{0,1\}\text{ for all }\alpha\in\Phi^+_{\overline1
        }\right\}
    \end{equation*}
\end{prop}
\begin{proof}
    For simplicity's sake, we define
    \begin{equation*}
        M:=\left\{\mathbf{m}\in\Z^{\tilde{\Phi}^+}\,\big|\, m_\alpha\in\{0,1\}\text{ for all }\alpha\in\Phi^+_{\overline1
        }\right\}.
    \end{equation*}
    We have that
    \begin{align*}
        S(\lambda+\mu|0^n)&=\hat{S}(\lambda+\mu)\cap M\\
        &=(\hat{S}(\lambda)+\hat{S}(\mu))\cap M\\
        &=(S(\lambda)+S(\mu))\cap M,
    \end{align*}
    where the second equality is Proposition \ref{prop-minkowski-decomp-chain}.
\end{proof}
Hence, we can now conclude
\begin{prop}
    Let $\lambda$ be a partition of length $l(\lambda)\leq m$. Then the set of vectors
    \begin{equation*}
        \left\{f^{\mathbf{s}}\cdot v_{(\lambda|0^n)}\mid\mathbf{s}\in S(\lambda|0^n)\right\}
    \end{equation*}
    is linearly independent in $\mathrm{gr}\,V(\lambda|0^n)$.
\end{prop}
\begin{proof}
    We decompose $S(\lambda|0^n)$ into the Minkowski sum
    \begin{equation*}
        S(\lambda|0^n)=\left(\sum_{i=1}^{l(\lambda)}(\lambda_i-\lambda_{i+
        1})S(1^i|0^n)\right)\cap\left\{\mathbf{m}\in\Z^{\tilde{\Phi}^+}\,\big|\, m_\alpha\in\{0,1\}\text{ for all }\alpha\in\Phi^+_{\overline1
        }\right\},
    \end{equation*}
    where $\lambda_{l(\lambda)+1}:=0$ and $kS(1^i|0^n)$ denotes the $k$-fold Minkowski sum of $S(1^i|0^n)$ with itself.
    Now, we know from section \ref{sec-indep-column-strip}, that our set of vectors for exterior powers $\bigwedge^i\C^{n|m}$ consists of \textit{essential monomials} (see Definition \ref{def-essential-mon-covariant}) for our monomial order, which is in particular homogeneous and thus refines the PBW filtration.
    Applying the same arguments as in \cite[Proposition 3.12]{Flag-Degen}, we conclude that the vectors $f^\mathbf{s}$ for $\mathbf{s}\in S(\lambda|0^n)$ are essential for $V(\lambda|0^n)$, and hence, linearly independent in the associated graded module $\mathrm{gr}\,V(\lambda|0^n)$.
\end{proof}
\section{Going down the hook: The basis for $(\lambda|\mu,0^{n-1})$}\label{sec-hook}
In this section, we extend our results to partitions of the form $(\lambda|\mu,0^{n-1})$, where $\lambda=(\lambda_1\geq\cdots\geq\lambda_{m-1}\geq \lambda_m)$ with $\lambda_m\neq0$ and $\mu\in\N$. We want to begin by providing a picture of the new pattern that we impose for $\mathfrak{gl}(3|2)$
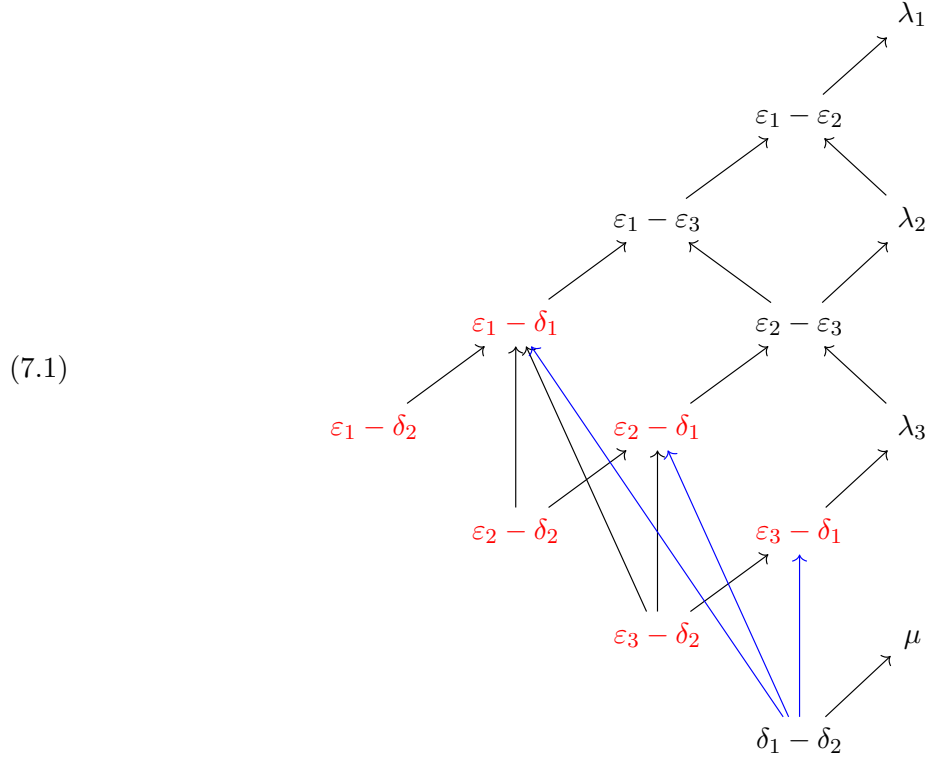
\begin{equation}\label{eq-fflv-pattern-hook-3|2}
\begin{tikzcd}[column sep=1em, row sep=2em]
	&&&& {\lambda_1} \\
	&&& {\e_1-\e_2} \\
	&& {\e_1-\e_3} && {\lambda_2} \\
	& {\textcolor{red}{\e_1-\d_1}} && {\e_2-\e_3} \\
	{\textcolor{red}{\e_1-\d_2}} && {\textcolor{red}{\e_2-\d_1}} && {\lambda_3} \\
	& {\textcolor{red}{\e_2-\d_2}} && {\textcolor{red}{\e_3-\d_1}} \\
	&& {\textcolor{red}{\e_3-\d_2}} && \mu \\
	&&& {\d_1-\d_2}
	\arrow[from=2-4, to=1-5]
	\arrow[from=3-3, to=2-4]
	\arrow[from=3-5, to=2-4]
	\arrow[from=4-2, to=3-3]
	\arrow[from=4-4, to=3-3]
	\arrow[from=4-4, to=3-5]
	\arrow[from=5-1, to=4-2]
	\arrow[from=5-3, to=4-4]
	\arrow[from=5-5, to=4-4]
	\arrow[from=6-2, to=4-2]
	\arrow[from=6-2, to=5-3]
	\arrow[from=6-4, to=5-5]
	\arrow[from=7-3, to=4-2]
	\arrow[from=7-3, to=5-3]
	\arrow[from=7-3, to=6-4]
	\arrow[color=blue, from=8-4, to=4-2]
	\arrow[color=blue, from=8-4, to=5-3]
	\arrow[color=blue, from=8-4, to=6-4]
	\arrow[from=8-4, to=7-5]
\end{tikzcd}.
\end{equation}
We use the usual set of inequalites as for our polytopes for the weights $(\lambda|0^2)$. But now, we add the following set of inequalities
\begin{align*}
    \mu\cdot(s_{\e_1-\d_1}+s_{\e_1-\e_3}+s_{\e_1-\e_2})+s_{\d_1-\d_2}\leq\lambda_1\cdot\mu\\
    \mu\cdot(s_{\e_2-\d_1}+s_{\e_2-\e_3}+s_{\e_1-\e_3}+s_{\e_1-\e_2})+s_{\d_1-\d_2}\leq\lambda_1\cdot\mu\\
    \mu\cdot(s_{\e_2-\d_1}+s_{\e_2-\e_3})+s_{\d_1-\d_2}\leq\lambda_2\cdot\mu\\
    \mu\cdot(s_{\e_3-\d_1})+s_{\d_1-\d_2}\leq\lambda_3\cdot\mu\\
    s_{\d_1-\d_2}\leq\mu
\end{align*}
We provide an interpretation of these inequalities later.

First, we define our set of inequalities for general $\gl(m|n)$:

We need to distinguish between extended Dyck paths ending in an even root or in the odd root $\e_m-\d_1$.
\begin{enumerate}
    \item{
        For every extended Dyck path of the form $\mathbf{p}=(p(0),\dots,p(u))$ such that $p(u)=\e_j-\e_{j+1}$ and $p(1)=\e_i-\d_l$ and $p(0)=\e_i-\d_k$ for some $1\leq l<k\leq n$  and $1\leq i\leq m$ and $1\leq j<m$, we set the conditions
        \begin{equation}\label{eq-hook-path-condition-generic}
            \mu\cdot\left(\sum_{t=1}^{u}s_{p(t)}\right)+s_{\d_1-\d_k}\leq \mu\cdot\lambda_j.
        \end{equation}
    }
    \item{
        For every extended Dyck path of the form $\mathbf{p}=(\e_m-\d_k,\dots,\e_m-\d_1)$, we set the conditions
        \begin{equation}\label{eq-hook-path-condition-edge}
            \mu\cdot\left(\sum_{t=1}^{u}s_{p(t)}\right)+s_{\d_1-\d_k}\leq \mu\cdot\lambda_m.
        \end{equation}
    }
\end{enumerate}

These conditions can be interpreted as follows: 
If along the extended Dyck path $\mathbf{p}$, the sum of the coordinates of the point $\mathbf{s}$ excluding $s_{p(0)}$ is equal to $\lambda_i$, then $s_{\d_1-\d_k}=0$.

Before we state our next result, we need to extend our subset of positive roots $\tilde{\Phi}^+$.

We define the new subset $\hat{\Phi}^+\subseteq\Phi^+$, where
\begin{align*}
    \hat{\Phi}^+_{\overline{0}}&=\left\{\e_i-\e_j\,\mid\,1\leq i< j\leq m \right\}\cup \{\delta_1-\delta_r\,\mid\, 1<r\le n\}\\
    \hat{\Phi}^+_{\overline{1}}&=\Phi^+_{\overline{1}}.
\end{align*}
\begin{defn}
    Let $\lambda$ be a partition of length $l(\lambda)\leq m$ with $\lambda_m\neq0$ and $\mu\in\N$. We define the polytope $P(\lambda|\mu,0^{n-1})\subseteq\mathbb{R}^{\hat{\Phi}^+}_{\geq0}$ subject to the conditions \eqref{eq-inequality-fflv-strip-even-dyck-path}, \eqref{eq-inequality-fflv-strip-extended-dyck-path}, \eqref{eq-inequality-fflv-strip-odd-root-bound-1} as well as \eqref{eq-hook-path-condition-generic} and \eqref{eq-hook-path-condition-edge}.
    
    We denote the set of lattice points of $P(\lambda|\mu,0^{n-1})$ by $S(\lambda|\mu,0^{n-1}):=P(\lambda|\mu,0^{n-1})\cap\Z^{\hat{\Phi}^+}$.
\end{defn}
\begin{rem}
    Similarly, as in Remark \ref{rem-strip-typical-matches-fourier-kus}, if $\lambda_m\geq n$, then the conditions \eqref{eq-hook-path-condition-generic} and \eqref{eq-hook-path-condition-edge} become redundant and we retrieve the polytope from \cite[Theorem 1.1]{Fourier-Kus}
\end{rem}
\begin{thm}\label{thm-fflv-big-hook}
    Let $\lambda$ be a partition of length $l(\lambda)\leq m$ with $\lambda_m\neq0$ and $\mu\in\N$. Then the set
    \begin{equation*}
        \left\{f^{\mathbf{s}}\cdot v_{(\lambda|\mu,0^{n-1})}\middle|\ \mathbf{s}\in S(\lambda|\mu,0^{n-1})\right\}
    \end{equation*}
    provides a basis of the associated graded space $\gr V(\lambda|\mu,0^{n-1})$.
\end{thm}
We want to proceed as in the case for $(\lambda|0^n)$ and verify three things:
\begin{enumerate}
    \item{
        Prove a straightening relation implying that our set of lattice points provides us a generating set of $\gr V(\lambda|\mu,0^{n-1})$.
    }
    \item{
        Check that the set of lattice points gives us a basis for the weights $(1^m|\mu,0^{n-1})$, which correspond to exterior powers $\bigwedge^{m+\mu}\C^{m|n}$.
    }
    \item{
        Check that the lattice points $S(\lambda|\mu,0^{n-1})$ can be decomposed into the Minkowski sum $S(\lambda-1^m|0^{n})+S(1^m|\mu,0^{n-1})$.
    }
\end{enumerate}
\subsection{The generating set for $(\lambda|\mu,0^{n-1})$}
We show that monomials breaking the conditions \eqref{eq-hook-path-condition-generic} or \eqref{eq-hook-path-condition-edge} can be straightened using derivations.

We begin by considering the extended Dyck path $\mathbf{p}=(p(0),\dots,p(u))$ with $p(1)=\e_i-\d_l$ and $p(0)=\e_i-\d_k$. Without loss of generality, we assume $p(u)=\e_1-\e_2$.
Then using notation from Proposition \ref{prop-straightening}, we know that we have
\begin{equation*}
    f^{\mathbf{s}'}+\sum_{\mathbf{s}\succ\mathbf{t}}c_{\mathbf{t}}f^{\mathbf{t}}\in I(\lambda),
\end{equation*}
where $\mathbf{s}'=\mathbf{s}-e_{\d_1-\d_k}+e_{\e_i-\d_k}$.

We now claim
\begin{claim}
    We have that
    \begin{equation*}
        \partial_{\e_i-\d_1}\left(f^{\mathbf{s}'}+\sum_{\mathbf{s}\succ\mathbf{t}}c_{\mathbf{t}}f^{\mathbf{t}}\right)=f^{\mathbf{s}}+\sum_{\mathbf{s}\succ\mathbf{t'}}c_{\mathbf{t'}}f^{\mathbf{t'}}\in I(\lambda|\mu,0^{n-1})
    \end{equation*}
\end{claim}
\begin{proof}
    We apply the differential $\partial_{\e_i-\d_1}$ to each term $f^\mathbf{t}$ and observe what monomials the actions on every variable yield us. 
    If $f^{\mathbf{t}}$ contains only even root vectors of the form $f_{\e_a-\e_b}$, then $\partial_{\e_i-\d_1}(f^{\mathbf{t}})=0$. In case we apply $\partial_{\e_i-\d_1}$ on the variable $f_{\e_a-\d_b}$, then this is begin replaced by $f_{\e_a-\e_i-\d_b+\d_1}$. 
    We distiguish now on $a$ and $b$:
    \begin{description}
        \item[$a=i$ and $b<k$]{
            We replace $f_{\e_a-\d_b}$ by the root vector $f_{\d_1-\d_b}$ which is smaller than $f_{\d_1-\d_k}$ in our order.
        }
        \item[$a=i$ and $b=k$]{
            We get the monomial $f^{\mathbf{t}-e_{\e_i-\d_k}}f_{\d_1-\d_k}$ and this monomial is still smaller than $\partial_{\e_i-\d_1}(f^{\mathbf{s}'})=f^{\mathbf{s}}$ as otherwise we would have $f^\mathbf{s'}<f^\mathbf{t}$.
        }
        \item[$a\neq i$ and $b=1$]{
            We get the variable $f_{\e_a-\e_i}$ or zero depending on if $a>i$. In the case $a>i$, we still have $f_{\e_a-\e_i}<f_{\d_1-\d_k}$.
        }
    \end{description}
\end{proof}
\subsection{The basis for exterior powers $\bigwedge^{m+\mu}\C^{m|n}$}
We now want to show that our set of lattice points provides us with a basis in the case of the weights $(1^m|\mu,0^{n-1})$ which are the exterior powers $\bigwedge^{m+\mu}\C^{m|n}$.
\begin{claim}
    The set
    \begin{equation*}
        \left\{f^{\mathbf{s}}\cdot v_{(\lambda|\mu,0^{n-1})}\middle|\mathbf{s}\in S(1^{m}|\mu,0^{n-1})\right\}
    \end{equation*}
    is linearly independent in $\gr V(1^m|\mu,0^{n-1})$.
\end{claim}
\begin{proof}\phantom{\qedhere}
    Similarly to the case for exterior powers $\bigwedge^k\C^{m|n}$ for $k\leq m$ in Lemma \ref{lem-bijection-small-fundamentals}, we want to define a specific set of functions parameterizing a weight basis of $\bigwedge^{m+\mu}\C^{m|n}$. 
    
    We define the set $B_{m+\mu}$ to be the set of functions $f:[m+\mu]\rightarrow K$ satisfying the following conditions:
    \begin{itemize}
        \item{
            If $f(i)\in[m]$, then $f(i)=i$.
        }
        \item{
            For $i< j$ such that $\{f(i), f(j)\}\subset K\setminus [k]$, either $f(i)>f(j)$ or $f(i)=f(j)\in \overline{[n]}$.
        }
    \end{itemize}
    
    Using similar arguments, we see that $B_{m+\mu}$ parametrises a weight basis of $\bigwedge^{m+\mu}\C^{m|n}$ under the mapping.
    \begin{equation*}
        f\mapsto e_{f(1)}\wedge\cdots\wedge e_{f(m+\mu)}
    \end{equation*}

    Hence, we reduce the claim to
    \begin{lem}
        The lattice points $S(1^m|\mu,0^{n-1})$ of $P(1^m|\mu,0^{n-1})$ are in bijection with $B_{m+\mu}$.
    \end{lem}
    \begin{proof}[Proof of Lemma]
        We begin by constructing a lattice point $\mathbf{s}^f\in S(1^m|\mu,0^{n-1})$ from a function $f\in B_{m+\mu}$. 
        If we denote the bijection from Lemma \ref{lem-bijection-small-fundamentals} for $B_m$ by $\varphi:B_m\rightarrow S(1^m|0^n)$, then we set $\mathbf{s}^f$ as follows
        \begin{align*}
            s^f_{\e_a-\d_b}&=\varphi(f\vert_{[m]})_{\e_a-\d_b}\\
            s^f_{\d_1-\d_c}&=\#(f\vert_{[m+1,\dots, m+\mu]}^{-1}(\overline{c})).
        \end{align*}

        We see that this is well-defined as follows: If $s^f_{\e_a-\d_b}=1$, then we have that $f(a)=\overline{b}$. 
        This means that $f(m+i)\leq \overline{b}$ for all $\overline{i}\in[\overline{n}]$. This means however that $s^f_{\d_1-\d_i}=0$ for all $i>b$ which is exactly the conditions defining $P(1^m|\mu,0^{n-1})$.

        To get the inverse construction, let $\mathbf{s}\in S(1^m|\mu,0^{n-1})$ be a lattice point. 
        We construct the function $f^\mathbf
        s\in B_{m+\mu}$ as follows:
        We firstly define for all $1\leq i\leq m$
        \begin{align*}
            f(i)=\varphi^{-1}(\mathbf{s}')(i),
        \end{align*}
        where $\mathbf{s'}\in S(1^m|0^n)$ is the projection of $\mathbf{s}$ onto $\Z^{\tilde{\Phi}^+}$.

        For the remaining elements, we need more notation:
        Let $n\geq i_1>\dots>i_r> 1$ be
        such that $s_{\d_1-\d_i}\neq0$ if and only $i=i_t$ for some $1\leq t\leq r$.

        Then, we define for $1\leq i\leq \mu$
        \begin{equation*}
            f(m+i)=\begin{cases}
                i_t\hspace{1em}&\text{ if } \sum\limits_{p=1}^{t-1}s_{\d_1-\d_{i_p}}<i\leq\sum\limits_{p=1}^{t}s_{\d_1-\d_{i_p}}\\
                \overline{1}&\text{ if }\sum\limits_{p=1}^{r}s_{\d_1-\d_{i_p}}<i
            \end{cases}
        \end{equation*}
        By construction, this function is in $B_{m+\mu}$ and hence yields us the desired bijection.
    \end{proof}
\end{proof}
\subsection{The Minkowski decomposition}
In order for us to get the linear independence for all weights of the form $(\lambda|\mu,0^{n-1})$ we prove the following claim.
\begin{prop}\label{prop-minkowski-big-hook}
    Let $\lambda$ be a partition of length $l(\lambda)\leq m$ with $\lambda_m\neq0$ and $\mu\in\N$. We have the decomposition
    \begin{equation*}
        S(\lambda|\mu,0^{n-1})=\left(S(\lambda-1^m|0^n)+S(1^m|\mu,0^{n-1})\right)\cap\left\{\mathbf{m}\in\Z^{\hat{\Phi}^+}\,\big|\, m_\alpha\in\{0,1\}\text{ for all }\alpha\in\Phi^+_{\overline1
        }\right\}.
    \end{equation*}
\end{prop}
\begin{proof}
    We show the two containments individually:
    \begin{itemize}
        \item['$\supseteq$']{
            Let $\mathbf{s}^1\in S(\lambda-1^m|0^n)$ and $\mathbf{s}^2\in S(1^m|\mu,0^{n-1})$ such that $s^1_{\e_i-\d_j}+s^2_{\e_i-\d_j}\leq 1$ for all $1\leq i\leq m$ and $1\leq j\leq n$.

            Then along every extended Dyck path $\mathbf{p}$ ending in $\e_i-\e_{i+1}$ or $\e_n-\d_1$, the sum  of the values of $\mathbf{s}^1$ on $\mathbf{p}$ is at most $\lambda_i-1$ or $\lambda_n-1$, respectively. 
            Hence adding the variables corresponding to $\d_1-\d_j$ from $\mathbf{s}^2$ does not violate the conditions \eqref{eq-hook-path-condition-edge} and \eqref{eq-hook-path-condition-generic} of $P(\lambda|\mu,0^{n-1})$.
        }
        \item['$\subseteq$']{
        Let $\mathbf{s}\in S(\lambda|\mu,0^{n-1})$ be a lattice point. We want to decompose $\mathbf{s}$ into
        \begin{equation*}
                \mathbf{s}=\mathbf{s}^1+\mathbf{s}^2,
        \end{equation*}
        where $\mathbf{s}^1\in S(\lambda-1^m|0^n)$ and $\mathbf{s}^2\in S(1^m|\mu,0^{n-1})$.

        We construct $\mathbf{s}^2\in S(1^m|\mu,0^{n-1})$ and show that $\mathbf{s}^1=\mathbf{s}-\mathbf{s}^2\in S(\lambda-1^m|0^n)$.
        
        Of course we have to set
        \begin{equation*}
            s^2_{\d_1-\d_j}=s_{\d_1-\d_j}\hspace{1em}\text{ for all }1<j\leq n.
        \end{equation*}
        The relevant part is for which odd roots $\e_a-\d_b$ we have to set $s^2_{\e_a-\d_b}=1$.

        The idea is going to be as follows: Among all odd roots $\e_a-\d_b$ with $s_{\e_a-\d_b}$=1, we want to choose increasing disjoint level sets starting from the bottom that would yield us an element $f\in B_m$ and do that until the level $1< i\leq n$ which is maximal such that $s_{\d_1-\d_i}\neq0$.

        We set
        \begin{equation*}
            t=\max\{1< r\leq n|s_{\d_1-\d_r}\neq 0\}
        \end{equation*}
        We now inductively define sets $J_i$ for $n\geq i\geq t$ as follows:

        For $i=n$, choose the set
        \begin{equation*}
            J_n=\{1\leq j\leq m|s_{\e_j-\d_n}=1\}
        \end{equation*}
        Assume then that $J_{k}$ has been constructed for all $k>i$ and that at least one of them is non-empty. 
        Let
        \begin{equation*}
            p=\max\{r>i|J_r\neq\emptyset\}
        \end{equation*}
        Then we set
        \begin{equation*}
            J_{i}=\{\max J_{p}< j\leq m| s_{\e_j-\d_i}=1\}
        \end{equation*}
        If $J_{k}=\emptyset$ for all $k>i$, then we set
        \begin{equation*}
            J_i=\{1\leq j\leq m| s_{\e_j-\d_i}=1\}
        \end{equation*}
        
        If we now set
        \begin{equation*}
            s^2_{\e_{j_i}-\d_i}=1\hspace{1em}\text{ for all }j_i\in J_i\text{ and }n\geq i\geq t,
        \end{equation*}
        then we see that $\mathbf{s}^2\in S(1^m|\mu,0^{n-1})$.

        Finally, we need to verify that $\mathbf
        {s}^1=\mathbf{s}-\mathbf{s}^2\in S(\lambda-1^m|0^n)$.

        Firstly, we do not need to be concerned with the FFLV-inequalites dealing with the even roots $\e_a-\e_b$, as the bounds $\lambda_i-1-(\lambda_j-1)=\lambda_i-\lambda_j$ remain unaffected.

        The crucial insight is the following: For every extended Dyck path $\mathbf{p}=(p(0),\dots,p(u))$ of maximal length such that there exists a $p(0)=\e_a-\d_b$ with $s_{\e_a-\d_b}=1$, we have that $n\geq b\geq t$ and $p(0)=\e_{j_b}-\d_b$ for some $j_b\in J_b$. 
        
        However, since we, by construction, remove exactly this odd root at the beginning of $\mathbf{p}$ from $\mathbf{s}$, we have that
        \begin{equation*}
            \sum_{r=1}^{u}{s_{p(r)}}\leq\lambda_c-1,
        \end{equation*}
        where $p(u)=\e_c-\e_{c+1}$ or $c=m$ and $p(u)=\e_c-\d_1$.
        Hence, we see that indeed $\mathbf{s}^1\in S(\lambda-1^m|0^n)$.
        }
    \end{itemize}
\end{proof}
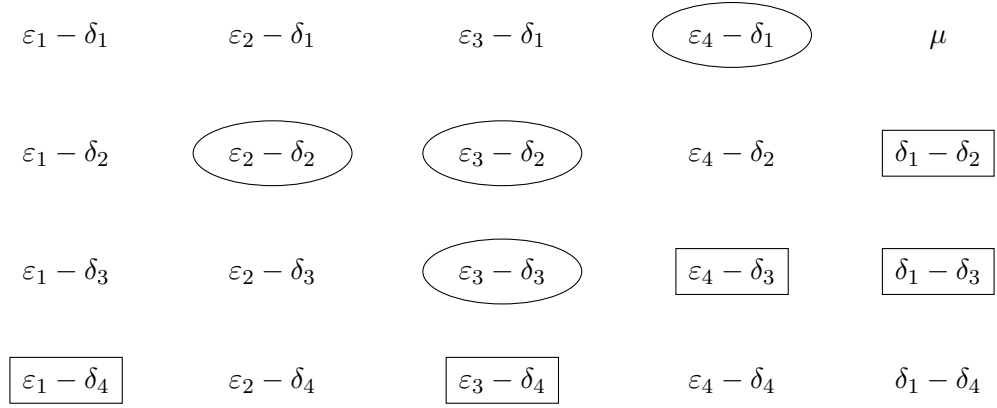
\begin{figure}
    $
\begin{tikzcd}[ampersand replacement=\&]
	|[alias=e1d1]| {\e_1-\d_1} \& |[alias=e2d1]| {\e_2-\d_1} \& |[alias=e3d1]| {\e_3-\d_1} \& |[alias=e4d1, draw, ellipse]| {\e_4-\d_1} \& |[alias=mu]| \mu \\
	|[alias=e1d2]| {\e_1-\d_2} \& |[alias=e2d2, draw, ellipse]| {\e_2-\d_2} \& |[alias=e3d2, draw, ellipse]| {\e_3-\d_2} \& |[alias=e4d2]| {\e_4-\d_2} \& |[alias=d1d2, draw, rectangle]| {\d_1-\d_2} \\
	|[alias=e1d3]| {\e_1-\d_3} \& |[alias=e2d3]| {\e_2-\d_3} \& |[alias=e3d3, draw, ellipse]| {\e_3-\d_3} \& |[alias=e4d3, draw, rectangle]| {\e_4-\d_3} \& |[alias=d1d3, draw, rectangle]| {\d_1-\d_3} \\
	|[alias=e1d4, draw, rectangle]| {\e_1-\d_4} \& |[alias=e2d4]| {\e_2-\d_4} \& |[alias=e3d4, draw, rectangle]| {\e_3-\d_4} \& |[alias=e4d4]| {\e_4-\d_4} \& |[alias=d1d4]| {\d_1-\d_4}
\end{tikzcd}$
    \caption{Example of how the Minkowski decomposition works in $\gl(4|4)$ from Proposition \ref{prop-minkowski-big-hook}. For any root $\alpha$ above, we have $s_{\alpha}\neq 0$ if and only if $\alpha$ is in an oval or rectangle. 
    We set $s^2_{\e_a-\d_b}=1$ if and only if $\e_a-\d_b$ is in a rectangle. Similarly, we have $s_{\d_1-\d_j}\neq 0$ if and only if $\d_1-\d_j$ is in a rectangle. Hence, we have $t=3$, $J_4=\{1,3\}$ and $J_3=\{4\}$.}
    \label{fig:placeholder}
\end{figure}

\begin{rem}
    Contrary to Proposition \ref{prop-minkowski-decomp}, we do not have such a Minkowski decomposition for all possible weights of the form $(\lambda|\mu,0^{n-1})$ and $(\tilde{\lambda}|\tilde{\mu},0^{n-1})$.

    For example we have for $\gl(1|2)$
    \begin{equation*}
        \left(S(1|1,0)+S(1|1,0)\right)\cap\left\{\mathbf{m}\in\Z^{\hat{\Phi}^+}\,\big|\, m_{\e_1-\d_1},m_{\e_1-\d_2}\in\{0,1\}\right\}\subsetneq S(2|2,0).
    \end{equation*}
    The number of lattice points on the left-hand-side is $10$, whereas the number of lattice points of $S(2|2,0)$ on the right-hand side is $12$.
\end{rem}

\section{Geometric applications: projective embeddings of flag supervarieties}\label{sec-degen}

Fix a weight $\nu\in \mathfrak h^*$, and let $v_\nu\in V(\nu)$ be a highest weight vector. Let $P$ be the parabolic subgroup in $G$ stabilizing the line $\C v_\nu\subset V(\nu)$. Then the orbit $G\cdot [v_\nu]$ gives an embedding of the partial flag supervariety $G/P\subset \P V(\nu)$. The goal of this section is to study the geometric properties of this embedded flag supervariety.\\

For each $k\ge0$, define $I_k\subset \mathrm{Sym}^k\big(V(\nu)^*\big)\simeq \mathrm{Sym}^k\big(V(\nu)\big)^*$ to be the largest $\g$-stable subspace orthogonal to $v_\nu^{ k}\in \mathrm{Sym}^k\big(V(\nu)\big)$. We claim that $I_k\cdot \mathrm{Sym}^\ell\big(V(\nu)^*\big)\subset I_{k+\ell}$.

To see this, pick $w\in I_k$ and $f\in \mathrm{Sym}^\ell\big(V(\nu)^*\big)$. Clearly, $\langle wf, v_\nu^{k+\ell}\rangle = \langle w, v_\nu^k\rangle \langle f, v_\nu^\ell\rangle = 0$. It therefore suffices to show that $x\cdot (wf)\in (v_\nu^k)^\perp$ for each $x\in \g$. However,

\begin{align*}
    \langle x\cdot (wf), v_\nu^{k+\ell}\rangle &= \langle (x\cdot w)f, v_\nu^{k+\ell}\rangle +(-1)^{|x|\cdot|w|}\langle w(x\cdot f), v_\nu^{k+\ell}\rangle\\
    &= \langle x\cdot w, v_\nu^{k}\rangle\langle f, v_\nu^\ell\rangle +(-1)^{|x|\cdot|w|}\langle w, v_\nu^k\rangle\langle x\cdot f, v_\nu^{k+\ell}\rangle\\
    &=0,
\end{align*}
where both terms are $0$ because $x\cdot w$ and $w$ both belong to $I_k\subset (v_\nu^k)^\perp$.

As a result,

\[I:= \bigoplus_{k=0}^\infty I_k\subset \mathrm{Sym}\big(V(\nu)^*\big) = \C[\P V(\nu)]\]
is a homogeneous ideal inside the homogeneous coordinate ring of $\P V(\nu)$.

\begin{prop}
    $G\cdot [v_\nu]$ is a closed subvariety in $\P V(\nu)$ cut out by the homogeneous ideal $I$.
\end{prop}
\begin{proof}
    We recall that in \cite[Theorem 1.3]{Gruson}, the analogous statement was proven for the lowest weight orbit $G\cdot[v_{-\nu}^*]\subset\P V(\nu)^*$ where the ideal $I\subset\mathrm{Sym}(V(\nu))$ is again graded and each graded component $I_k\subset\mathrm{Sym}^k(V(\nu))$ is the largest $\g$-stable subspace orthogonal to $(v_{-\nu}^*)^k\in\mathrm{Sym}^k(V(\nu)^*)$. We apply this result to get our claim:

    We consider the opposite Borel subalgebra $\mathfrak{b}^{\mathrm{op}}\subseteq\gl(m|n)$ of lower triangular matrices. Then $v_\nu\in V(\nu)$ is a $\mathfrak{b}^{\mathrm{op}}$-lowest weight vector.

    Now we know that $V(\nu)^*=V_{\mathfrak{b^{\mathrm{op}}}}(\eta)$, i.e. the dual space $V(\nu)^*$ is a highest weight space with respect to the opposite Borel $\mathfrak{b}^{\mathrm{op}}$ with highest weight $\eta\in\mathfrak{h}^*$.

    If we now apply \cite[Theorem 1.3]{Gruson} to $\eta$ and $\mathfrak{b}^{\mathrm{op}}$ , we get our claim.
\end{proof}

Now assume that $\nu$ is covariant. Then we have the following alternate description of $I_k$.

\begin{prop}
    If $\nu$ is covariant, then for each $k\ge 0$ we get a decomposition
    \[\mathrm{Sym}^k\big(V(\nu)^*\big) = I_k \oplus V(k\nu)^*\]
    as $\g$-representations.
\end{prop}

\begin{proof}
    Begin by noting that if $\nu$ is covariant, then $V(\nu)^{\otimes k}$ is semisimple, and hence so is $\mathrm{Sym}^k\big(V(\nu)^*\big)\subset \left(V(\nu)^*\right)^{\otimes k}$. 
    Picking a weight basis for $V(\nu)$ and letting $v_\nu^*\in V(\nu)^*$ be the dual basis element to $v_\nu$, we see that $(v_\nu^*)^k\in \mathrm{Sym}\big(V(\nu)^*\big)$ generates a copy of $V(k\nu)^*$ by comparing their respective lowest weights, which in both cases is exactly $-k\nu$. 
    By semisimplicity, we get that
    \[\mathrm{Sym}^k\big(V(\nu)^*\big) = V(k\nu)^*\oplus M\]
    for some $\g$-submodule $M$. Moreover, $\langle (v_\nu^*)^k, v_\nu^k\rangle = 1\ne 0$, so that $M\subset (v_\nu^k)^\perp$. By the irreducibility of $V(k\nu)^*$, it follows that $M = I_k$ as required.
\end{proof}

\begin{cor}\label{cor-coordinate-ring}
    If $\nu$ is covariant, then the coordinate ring of the embedded partial flag supervariety $G/P\subset\P V(\nu)$ is isomorphic to
    \begin{equation*}
        \mathbb{C}[G/P]_\nu:=\bigoplus_{k\ge0}V(k\nu)^*,
    \end{equation*}
    where the multiplication is induced from the $\g$-equivariant map
    \begin{equation*}
        v^*_{m\nu}\otimes v^*_{n\nu}\mapsto v^*_{(n+m)\nu}
    \end{equation*}
\end{cor}
\begin{proof}
    For the last part, we see that the multipication is inherited from $\mathrm{Sym}(V(\nu)^*)$ and thus $\g$-equivariant. Hence, by definition, if we multiply $(v^*_{\nu})^m$ with $(v^*_{\nu})^n$, we get $(v^*_{\nu})^{n+m}$.
\end{proof}

We now recall the notion of essential monomials and favourableness from \cite{Flag-Degen} applied to the modules $V(k\nu)$ for $k\geq 0$.

\begin{defn}
    Fix a basis of $\mathfrak{n}^-$ given by $(f_1,\dots,f_n)$ and a monomial order $<$ on $\Sym(\mathfrak{n}^-)$. If $\nu$ is covariant, then we say a monomial $\mathbf{m}\in\Z^n$ is \textit{essential} for $V(\nu)$ if 
    \begin{equation*}
        f^\mathbf{m}\cdot v_{\nu}\not\in\mathrm{span}\langle f^\mathbf{n}\cdot v_\nu\mid\mathbf{n<\mathbf{m}}\rangle_\C
    \end{equation*}
\end{defn}
\begin{defn}\label{def-essential-mon-covariant}
    If $\nu$ is covariant, we say $V(\nu)$ is \textit{favourable} if every essential monomial of $V(k\nu)$ can be written as a sum of $k$ essential monomials of $V(\nu)$. 
\end{defn}

\begin{prop}\label{prop-strip-favourable}
    Fix a covariant weight of the form $\nu=(\nu^1|0^n)$. Then the module $V(\nu)$ is favourable with respect to the monomial order from Definition \ref{def-monomial-order}.
\end{prop}
\begin{proof}
    By construction of the monomials from Theorem \ref{thm-fflv-strip}, we know that they are essential. Then, the Minkowski sum property from Proposition \ref{prop-minkowski-decomp} provides us with the claim.
\end{proof}
\begin{rem}
    If our highest weight is of the form $(\lambda|\mu,0^{n-1})$, where $\mu\neq0$, then we cannot use the previous argument to get that the module $V(\lambda|\mu,0^{n-1})$ is favourable.  We suspect that the semigroup of essential monomials $\Gamma((\lambda|\mu,0^{n-1}),<)$ \cite[Definition 3.13]{Flag-Degen} is not even finitely generated. 
\end{rem}
Next, we want to see how we can relate the results from \cite{Flag-Degen} to the coordinate ring $\C[G/P]_\nu$. We first recall the ring considered there:

We consider for $k\ge 0$ the \textit{Kac modules} \cite{Kac_Rep} $K_\mathfrak{b}(k\nu)$, which are the maximal finite-dimensional highest weight modules of highest weight $k\nu$. Then we can put a superalgebra structure on 
\begin{equation*}
    \bigoplus_{k\ge 0}K_{\mathfrak{b}}(k\nu)^*.
\end{equation*}
The multiplication is given by dualizing the Cartan map, see \cite[Proposition 3.2]{Flag-Degen}
\begin{equation}\label{eq-Cartan-map-Kac}
    \begin{aligned}
    C_{m,n}:K_{\mathfrak{b}}((m+n)\lambda)&\rightarrow K_{\mathfrak{b}}(m\lambda)\otimes_{\C} K_{\mathfrak{b}}(n\lambda)\\
    v^{\mathrm{Kac}}_{(m+n)\lambda}&\rightarrow v^{\mathrm{Kac}}_{m\lambda}\otimes v^{\mathrm{Kac}}_{n\lambda},
\end{aligned}
\end{equation}
where $v^{\mathrm{Kac}}_{k\nu}$ refers to the highest weight vector of the Kac module $K_{\mathfrak{b}}(k\nu)$ and we use the following isomorphism of $\g$-modules \cite[Appendix A.2.3.]{Musson}
\begin{equation}\label{eq-swap-dual-tensor}
\begin{aligned}
    \gamma_{V,W}:W^*\otimes V^*&\rightarrow(V\otimes W)^*\\
    \varphi\otimes\psi&\mapsto\left(r\otimes s\mapsto\psi(r)\varphi(s)\right).
\end{aligned}
\end{equation}
\begin{lem}\label{lem-embedding-into-kacs}
    We have the embedding of superalgebras
    \begin{equation*}
        \C[G/P]_\nu:=\bigoplus_{k\ge 0}V(k\nu)^*\hookrightarrow\bigoplus_{k\ge 0}K_{\mathfrak{b}}(k\nu)^*
    \end{equation*}
\end{lem}
\begin{proof}
    We know that the Kac module $K_{b}(k\nu)$ surjects onto $V(k\nu)$ by mapping the corresponding highest 
    weight vectors $v^{\mathrm{Kac}}_{k\nu}\mapsto v_{k\nu}$, call this map $\varphi_k:K_{b}(k\nu)\rightarrow V(k\nu)$. Hence we have an injective map on the duals $\varphi_k^*:V(k\nu)^*\rightarrow K_{b}(k\nu)^*$.

    We consider the map
    \begin{equation*}
        \varphi:=\bigoplus_{k\ge 0}\varphi_k^*
    \end{equation*}

    We see that this map is $\g$-equivariant. 
    We only need to verify that this map is also compatible with the multiplication.

    For this, consider the following diagram
    \begin{equation*}
        \begin{tikzcd}[ampersand replacement=\&]
        	{V((k+l)\nu)^*} \&\& {K_{\mathfrak{b}}((k+l)\nu)^*} \\
        	\\
        	{V(k\lambda)^*\otimes V(l\lambda)^*} \&\& {K_{\mathfrak{b}}(k\lambda)^*\otimes K_{\mathfrak{b}}(l\lambda)^*}
        	\arrow["{\varphi^*_{k+l}}", from=1-1, to=1-3]
        	\arrow["{\mathrm{mult}_{\mathbb{C}[G/P]_\nu}}", from=3-1, to=1-1]
        	\arrow["{\varphi^*_{k}}\otimes\varphi^*_{l}",from=3-1, to=3-3]
        	\arrow["{C_{k,l}^*}", from=3-3, to=1-3]
        \end{tikzcd}
    \end{equation*}
    We want this diagram to commute, which we can check by dualizing the maps and considering
    \begin{equation*}
        \begin{tikzcd}[ampersand replacement=\&]
        	{V((k+l)\nu)} \&\& {K_{\mathfrak{b}}((k+l)\nu)} \\
        	\\
        	{V(l\lambda)\otimes V(k\lambda)} \&\& {K_{\mathfrak{b}}(l\lambda)\otimes K_{\mathfrak{b}}(k\lambda)}
        	\arrow["{\mathrm{mult}_{\mathbb{C}[G/P]_\nu}}^*"', from=1-1, to=3-1]
        	\arrow["{\varphi_{k+l}}",from=1-3, to=1-1]
        	\arrow["{C_{k,l}}", from=1-3, to=3-3]
        	\arrow["{\varphi_{l}}\otimes\varphi_{k}",from=3-3, to=3-1]
        \end{tikzcd},
    \end{equation*}
    where we utilized \eqref{eq-swap-dual-tensor}.
    The dualization of the multiplication in $\C[G/P]_\nu$ is the Cartan embedding from Lemma \ref{lem-cartan-embedding}.

    Commutativity of this diagram follows from checking that it commutes for the highest weight vector of the Kac module $v^{\mathrm{Kac}}_{(k+l)\nu}$ and that the maps are $\g$-equivariant.
\end{proof}
As an application of this result, we get a degeneration of the embedded partial flag supervariety
\begin{thm}\label{thm-degeneration}
    There exists a flat proper morphism $\kappa:\mathrm{Proj}\mathcal{R}\rightarrow\mathbb{A}^{1|0}$ such that the fibers are
    \begin{equation*}
        \kappa^{-1}(a)\cong\begin{cases}
            \mathrm{Proj}(\C[G/P]_\nu) &a\neq0\\
            \mathrm{Proj}(\C[x^\mathbf{s}u\,|\,\mathbf{s}\in S(\nu^1|0^n)]) &a=0.
        \end{cases}
    \end{equation*}
    The affine spectrum of the algebra of the special fibre $\mathrm{Spec}\,\C[x^\mathbf{s}u\,|\,\mathbf{s}\in S(\nu^1|0^n)]$ is in addition a toric supervariety as defined by Jankowski \cite{Toric-Super-New,Toric-Super-One-Odd}.
\end{thm}
\begin{proof}
    We can use \cite[Theorem 1, or Corollary 7.2]{Flag-Degen} applied to $V(k\nu)^*$, to get a flat morphism of superalgebras $\C[t]\rightarrow\mathcal{R}$ such that the corresponding closed fibres $\mathcal{R}/\langle t-a\rangle$ are $\C[G/P]_\nu$ for $a\neq 0$ and $\C[x^\mathbf{s}u|\mathbf{s}\in S(\nu^1|0^n)]$ for $a=0$, respectively.

    The reason why we can do this is that by Lemma \ref{lem-embedding-into-kacs}  the multiplication in $\C[G/P]_\nu$ matches that of the ring of Kac modules from \eqref{eq-Cartan-map-Kac} and \eqref{eq-swap-dual-tensor}.

    In order to take projective spectra, we note the that generators $g_k$ considered in \cite[Equation (6)]{Flag-Degen} are homogeneous. 
    
    To those, we can find by the proof of \cite[Theorem 1]{Flag-Degen} a weight vector $w\in\Z^\Phi$ singling out binomials and monomials defining the monomial superalgebra considered in the special fibre.
    Since the generators are homogeneous, we can add a multiple of the vector $\mathbf{1}=(1,\dots,1)\in\Z^\Phi$ without changing any of the initial forms of these generators, i.e. $\mathrm{in}_{w+r\mathbf{1}}(g_k)=\mathrm{in}_{w}(g_k)$ for all $r\in\Z$.
    Hence, the graded algebra $\mathcal{R}$ \cite[16]{Flag-Degen} would have in its $0$-graded piece only $\C[t]$ and thus, we have a proper map to $\A^1$.

    Lastly, we can use \cite[Lemma 7.11]{Flag-Degen} to give us that the affine superscheme $\mathrm{Spec}\,\C[x^\mathbf{s}u\,|\,\mathbf{s}\in S(\nu^1|0^n)]$ is an affine toric supervariety.
\end{proof}
\printbibliography
\end{document}